\documentclass[a4paper,11pt]{article}
\usepackage[utf8]{inputenc}
\usepackage{amsmath}
\usepackage{amssymb}
\usepackage{amsthm}
\usepackage{xcolor}

\usepackage{comment}

\newcommand{\norm}[1]{\lvert\lvert#1\rvert\rvert}

\theoremstyle{plain}
\newtheorem{thm}{Theorem}[section]
\newtheorem{Prop}[thm]{Proposition}
\newtheorem{Def}[thm]{Definition}
\newtheorem{lma}[thm]{Lemma}
\newtheorem{cor}[thm]{Corollary}

\newtheorem{remark}[thm]{Remark}

\newtheorem*{example*}{Example}
\newtheorem*{remark*}{Remark}

\newcommand{\Ric}{\mathrm{Ric}}

\newcommand{\Lip}{\mathrm{Lip}}
\newcommand{\F}{\mathcal{F}}
\newcommand{\E}{\mathcal{E}}

\newcommand{\eps}{\varepsilon}

\newcommand{\cC}{\mathcal C}

\def\N{{\mathbb N}}

\def\R{{\mathbb R}}

\def\D{{\mathcal D}}

\title{Functional inequalities for the heat flow on time-dependent metric measure spaces}
\author{Eva Kopfer,
 \quad Karl-Theodor Sturm\quad
 \thanks{Institut f\"ur Angewandte Mathematik, Universit\"at Bonn, Endenicher Allee 60, 53115 Bonn, Germany (\texttt{eva.kopfer@iam.uni-bonn.de, sturm@uni-bonn.de})
 Both authors gratefully acknowledge  support by the German Research Foundation through the Hausdorff Center for Mathematics and the Collaborative Research Center 1060. The second author also gratefully acknowledges 
 support by the European Union through the ERC-AdG ``RicciBounds''.}
 }
\date{}

\begin{document}

\maketitle
\begin{abstract} 
We prove that  synthetic lower Ricci bounds for metric measure spaces -- both in the sense of Bakry-\'Emery and in the sense of Lott-Sturm-Villani --
can be characterized by various functional inequalities including local Poincar\'e  inequalities, local logarithmic Sobolev inequalities, dimension independent Harnack inequality, and logarithmic Harnack inequality.

More generally, these equivalences will be proven in the setting of time-dependent metric measure spaces 
and  will provide a characterization of super-Ricci flows  of metric measure spaces.
\end{abstract}

\tableofcontents

\section{Introduction}
\subsection{Setting}
Huge research interest and extensive literature is devoted to the study of functional inequalities for the heat equation, both on Riemannian manifolds and on more abstract spaces.
Of particular importance are  functional inequalities which are equivalent to a uniform lower bound on the Ricci curvature, say $\Ric_g\ge K\cdot g$.
In F.-Y.\ Wang's monograph \cite{wang2014analysis}, Theorem 2.3.3., an impressive collection of 15 equivalent properties is listed.

In principle, all these properties and equivalences should hold -- and indeed most of them do hold -- in much more general settings.
Many of them have been reformulated and proven in the setting of Markov diffusion semigroups and $\Gamma$-calculus, initiated by the seminal work of Bakry \& \'Emery \cite{bakryemery} and culminating now in the monograph \cite{Bakry} of Bakry, Gentil and Ledoux,
see 
Theorems 4.7.2, 5.5.2, 5.5.5, 5.6.1
and Remark 5.6.2 in \cite{Bakry}.

Another, more recent, important setting for the study of heat equations and functional inequalities are metric measure spaces, in particular, such mm-spaces which are infinitesimally Hilbertian and which satisfy a synthetic lower Ricci bound as introduced in the
foundational works of  Sturm \cite{sturm2006} and
  Lott\,\&\,Villani \cite{lott2009ricci}.
In a series of ground breaking papers, Ambrosio, Gigli\,\&\,Savar\'e \cite{agscalc, agsmet, agsbe} introduced and analyzed the heat flow on such spaces and derived various functional inequalities. In particular, they proved that both  the
Bochner inequality (without dimensional term) and the $L^2$-gradient estimate are equivalent to the synthetic Ricci bound CD$(K,\infty)$; and they deduced 
the local Poincar\'e inequality and the logarithmic Harnack inequality.
Savar\'e \cite{savare} extended the powerful self-improvement property of Bochner's inequality to mm-spaces and utilized it to deduce the $L^1$-gradient estimate;
based on the latter, H.~Li \cite{H.Li} proved the dimension-independent Harnack inequality which in turn implies the logarithmic Harnack inequality.

Only recently, some of these properties and equivalences have been extended to the heat flow on time-dependent Riemannian manifolds, e.g. by Cheng \& Thalmaier \cite{cheng2018evolution}, Haslhofer \& Naber \cite{HN2015}, McCann \& Topping \cite{mccanntopping},  and Cheng \cite{cheng2017diffusion}.
The authors of the current paper had been the first to study the heat flow on time-dependent metric measure spaces \cite{sturm2016}, to introduce the time-dependent counterpart of synthetic lower Ricci bounds, and to derive various functional inequalities equivalent to it.

\medskip

  Here and throughout this paper, the setting will be as follows.
$(X,d_t,m_t)_{t\in I}$ is a time-dependent metric measure space where $I=(0,T)$ and $X$ is a topological space.  The Borel measures $m_t=e^{-f_t}m$ and the geodesic distances $d_t$ are assumed to be  logarithmic Lipschitz continuous in time. Moreover, the maps $x\mapsto f_t(x)$ are assumed to be bounded and Lipschitz continuous. That is, there exists a constant $L>0$ such that for all $x,y\in X$ and $s,t\in I$
\begin{align}\tag{{\bf A1.a}}
|f_t(x)-f_s(y)|\leq L|t-s|+Ld_t(x,y), \qquad
\Big|\log\frac{d_t(x,y)}{d_s(x,y)}\Big|\leq L|t-s|.
\end{align}
Furthermore,  for some  $K\in\mathbb R$ and each $t\in I$  the static mm-space
\begin{align}\tag{{\bf A1.b}}\mbox{$(X,d_t,m_t)$  satisfies the  condition RCD$(K,\infty)$.}
\end{align}

The static mm-space $(X,d_t,m_t)$ 
defines a Dirichlet form $\E_t$, a Laplacian $\Delta_t$, and a square field operators $\Gamma_t$ related to each other via
\begin{align}-\int_X u\,\Delta_t v\,dm_t=\E_t(u,v)=\int_X\Gamma_t(u,v)\,dm_t\qquad \forall u\in\D(\E_t), v\in\D(\Delta_t).
\end{align}

 The domains $\D(\E_t)$ define Hilbert spaces with scalar products $\int uv\, dm_t+\E_t(u,v)$. Note that the scalar products are mutually equivalent, since we have uniform ellipticity by \textbf{(A1.a)}, i.e.
 \begin{align}\label{eq: s.f.}
 e^{-2L|t-s|}\Gamma_s(u)\leq \Gamma_t(u)\leq e^{2L|t-s|}\Gamma_s(u),
 \end{align}
 for some constant $C>0$, and for all $s, t\in I$.
 We fix an arbitrary $t$ and set $\F=\D(\E)$ as a reference Hilbert space.
 We have the dense and continuous embeddings $\F\subset L^2(X,m)\subset \F^*$, where $\F^*$ denotes the dual space of $\F$.
 Similarly $L^p(X)=L^p(X,m)$ will serve as a reference $L^p$-space.

The family of mm-spaces $(X,d_t,m_t)_{t\in I}$ defines a 2-parameter family of heat propagators $(P_{t,s})_{s\leq t}$ and adjoint propagators $(P^*_{t,s})_{s\leq t}$ on $L^2(X,m)$, see \cite{sturm2016} for details. 
The heat flow $t\mapsto u_t=P_{t,s}u$ provides solutions to the heat equation
\begin{align*}
\partial_tu_t=\Delta_tu_t \text{ on } (s,T)\times X \text{ with }u_s=u,
\end{align*}
  whereas 
 $s\mapsto P_{t,s}^*v$ provides solutions to  the adjoint heat equation
\begin{align*}
\partial_sv_s=\Delta_s^*v_s:=-\Delta_sv_s+v_s\dot f_s \text{ on } (0,t)\times X \text{ with }v_t=v.
\end{align*}

By duality, the propagator $(P_{t,s})_{s\leq t}$ acting on bounded continuous functions induces a dual propagator $(\hat P_{t,s})_{s\leq t}$ acting on probability measures as follows
\begin{align*}
\int u\, d(\hat P_{t,s}\mu)=\int P_{t,s} u\, d\mu \qquad \forall u\in \mathcal C_b(X), \forall \mu\in\mathcal P(X). 
\end{align*}

The  main result of our previous paper  
is the characterization of  super-Ricci flows of mm-spaces in terms of the heat flow on them. For $t\in I$, let  $W_t$ denote the $L^2$-Kantorovich-Wasserstein metric with respect to $d_t$ and let $S_t(\mu)=\int \log (d\mu/dm_t)\, d\mu$ denote the relative Boltzmann entropy with respect to $m_t$.

\begin{thm}[\cite{sturm2016}]\label{old}
The following assertions are equivalent:
\begin{enumerate}
\item[\bf(i)]  For a.e.\ $t\in (0,T)$ and every  $W_t$-geodesic  $(\mu^a)_{a\in[0,1]}$ in $\mathcal P(X)$ with
$\mu^0,\mu^1\in \D(S)$
\begin{equation}\label{est-I}\tag{{\bf E1}}
\partial_a S_t(\mu^{a})\big|_{a=1}-\partial_a S_t(\mu^{a})\big|_{a=0}
\ge- \frac 12\partial_tW_{t}^2(\mu^0,\mu^1).
\end{equation}
\item[\bf(ii)] For all $0<s<t<T$ and $\mu,\nu\in\mathcal P(X)$
\begin{equation}\label{est-II}\tag{{\bf E2}}
W_s (\hat P_{t,s}\mu,\hat P_{t,s}\nu)\le W_t (\mu,\nu)
\end{equation}
\item[\bf(iii)] For all $u\in\F$ and all $0<s<t<T$
\begin{equation}\label{est-III}\tag{{\bf E3}}
\Gamma_t(P_{t,s}u)\le P_{t,s}\big(\Gamma_s (u)\big)
\end{equation}
  \item[\bf(iv)] For all $0<s<t<T$ and for all $u_s,g_t\in\F$ with $g_t\ge0$,  $g_t\in L^\infty(X,m)$, $u_s\in \Lip(X)$
and for a.e.\ $r\in(s,t)$ 
 \begin{equation}\label{est-IV}\tag{{\bf E4}}
 {\bf \Gamma}_{2,r}(u_r)(g_r)
 \geq\frac12\int\stackrel{\bullet}{\Gamma}_r(u_r)g_rdm_r
\end{equation}
where $u_r=P_{r,s}u_s$ and $g_r=P^*_{t,r}g_t$.
\end{enumerate}
\end{thm}
Here
 $${\bf \Gamma}_{2,r}(u_r)(g_r):= \int\Big[\frac12\Gamma_{r}(u_r)\Delta_r g_r
 +(\Delta_r u_r)^2g_r
 +\Gamma_r(u_r,g_r)\Delta_ru_r\Big]dm_r$$ denotes the distribution valued $\Gamma_2$-operator (at time $r$) applied to $u_r$ and tested against $g_r$
 and
$$\stackrel{\bullet}{\Gamma}_r(u_r):=\mbox{w-}\lim_{\delta\to0}\
\frac1\delta\Big(\Gamma_{r+\delta}(u_r)-\Gamma_r(u_r)\Big)$$ denotes any subsequential weak limit  of $\frac1{2\delta}\big(\Gamma_{r+\delta}-\Gamma_{r-\delta}\big)(u_r)$ in $L^2((s,t)\times X)$.

We say that a one-parameter family of mm-spaces $(X,d_t,m_t)_{t\in I}$ is a  \emph{super-Ricci flow} -- or that it evolves as a super-Ricci flow -- if it satisfies one/each assertion of the previous Theorem. This is a canonical extension of the notion of \emph{super-Ricci flows of Riemannian manifolds} $(M,g_t)$ defined through the tensor inequality
\begin{align*}
\Ric_t\geq -\frac12\partial_tg_t.
\end{align*}

Property (i) above is called \emph{dynamic convexity} of the Boltzmann entropy. This concept has been introduced by the second author in \cite{sturm2015}; it provides a canonical generalization of the synthetic  Ricci bound CD$(0,\infty)$ defined in terms of the semiconvexity of the Boltzmann entropy in the static setting.

Property (iv) is the appropriate generalization of Bochner's inequality  or, in other words, of the Bakry-\'Emery condition to the time-dependent setting. It will be called \emph{dynamic Bochner inequality (integrated in time)}. 

In contrast to that, we say that the \emph{dynamic Bochner inequality pointwise in time} holds 
 if $\forall t\in I$, 
$\forall u,g\in\D(\Delta)\cap L^\infty(X,m)$ with  $\Gamma_t(u)\in L^\infty(X,m)$ and $g\ge0$
\begin{equation}\label{ptwBochner0}\tag{{\bf E5}}
\int \Big[\Gamma_t(u)\Delta_t g
+2(\Delta_t u)^2g
+ 2\Gamma_t(u,g)\Delta_t u
-\partial_t\Gamma_t(u)g\Big] dm_t\ge0.
\end{equation}

In the static case, Bochner's inequality has the remarkable and powerful `self-improvement property' which allows to deduce
improved versions of the  assertions in the previous Theorem, in particular, to derive the $L^1$-gradient estimate.
This self-improvement strategy  in the time-dependent case  requires additional time regularity of the involved quantities.
It was carried out by the first author in \cite{Ko2} and 
 can be reformulated with the notation from the current paper as follows.
 
\begin{thm}[\cite{Ko2}]\label{eva} Assume ({\bf A2.a+c}), see Section 2. Then 
the  $L^2$-gradient estimate \eqref{est-III} is equivalent to the 
 $L^1$-gradient estimate:
 for all $u\in F$ and all $0<s<t<T$
\begin{equation}\label{L1-grad}\tag{{\bf E6}}
\big(\Gamma_t(P_{t,s}u)\big)^{1/2}\le P_{t,s}\big(\Gamma_s (u)^{1/2}\big)
\end{equation}
 Moreover, the dynamic Bochner inequality (integrated in time) implies the dynamic Bochner inequality pointwise in time which in turn implies the $L^1$-gradient estimate as formulated above.
\end{thm}

Additional assumptions on time regularity (e.g. continuity of $t\mapsto \Delta_t P_{t,s}u$ in appropriate spaces) will be also requested for various results of the current paper; we will formulate these assumptions tailor-made in the subsequent sections.

\subsection{Summary of the main results}

Let us summarize the main results of the current paper.
To simplify and unify the presentation here in the introduction, we will restrict ourselves to the case $m_t(X)<\infty$ and in addition to  our standing assumptions ({\bf A1.a+b}) we will 
%
request now all the assumptions which ever will be made in the sequel. 
Besides our standing assumptions ({\bf A1.a+b}), these are assumptions ({\bf A2.a-c}) formulated in Section 2, ({\bf A3}) formulated in Section 3, and assumptions ({\bf A5.a+b}) formulated in Section 5.
  We emphasize that all these extra assumptions are always fulfilled in the static case and they are also satisfied in the case of Riemannian manifolds with metric tensors which smoothly depend on time.
  
\begin{thm}\label{mainthm} Under the previously mentioned assumptions, the following assertions are equivalent:
\begin{enumerate}
\item[\bf(i)] $(X,d_t,m_t)_{t\in I}$ is a super-Ricci flow.
\item[\bf(ii)] One/each of the local Poincar\'e inequalities holds
\begin{subequations}
\label{locpoinc}
 \begin{align}
  P_{t,s}(u^2)(x)-(P_{t,s}u)^2(x)\leq&2(t-s)P_{t,s}(\Gamma_su)(x) \label{locpoinc1}\tag{{\bf E7}}\\
  P_{t,s}(u^2)(x)-(P_{t,s}u)^2(x)\geq &2(t-s)\Gamma_t(P_{t,s}u)(x). \label{locpoinc2}\tag{{\bf E8}}
 \end{align}
 \end{subequations}
\item[\bf(iii)] One/each of the local logarithmic Sobolev inequalities holds
\begin{subequations}
\label{logsob}
\begin{align}
P_{t,s}(u\log u)-P_{t,s}u\log P_{t,s}u\leq& (t-s)P_{t,s}\left(\frac{\Gamma_s(u)}{u}\right),\label{logsob1}\tag{{\bf E9}}\\
P_{t,s}(u\log u)-P_{t,s}u\log P_{t,s}u\geq& (t-s)\frac{\Gamma_t(P_{t,s}u)}{P_{t,s}u}. \label{logsob2}\tag{{\bf E10}}
\end{align}
\end{subequations}
\item[\bf(iv)] The dimension independent Harnack inequality holds for one/each $\alpha\in (1,\infty)$
 \begin{align}\label{dimHarn}\tag{{\bf E11}}
  (P_{t,s}u)^\alpha(y)\leq P_{t,s}(u^\alpha)(x)\exp\left\{\frac{\alpha d_t^2(x,y)}{4(\alpha-1)(t-s)}\right\}.
 \end{align}
 \item[\bf(iv)] The
  logarithmic Harnack inequality holds
 \begin{align}\label{logHarn}\tag{{\bf E12}}
 P_{t,s}(\log u)(x)  \leq \log (P_{t,s}u)(y)+\frac{d_t^2(x,y)}{4(t-s)}.
 \end{align}
\end{enumerate}
The formulation ``one/each'' in particular means that \emph{one} of the respective properties implies \emph{each} of the respective properties.
\end{thm}

\begin{remark} 
a) Upper and lower local Poincar\'e inequalities together obviously imply the $L^2$-gradient estimate \eqref{est-III}.
Upper and lower local logarithmic Sobolev inequality together imply
\begin{align*}
\frac{\Gamma_t(P_{t,s}u)}{P_{t,s}u}\leq P_{t,s}\left(\frac{\Gamma_s(u)}{u}\right),
\end{align*}
which is a prioiri weaker than the $L^1$-gradient estimate \eqref{L1-grad}. Indeed the $L^1$-gradient estimate together with Jensen's inequality applied to the function $\beta(z,w)=z^2/w$ imply
\begin{align*}
 \frac{\Gamma_t(P_{t,s}u)}{P_{t,s}u}\le  \frac{\big(P_{t,s}\sqrt{\Gamma_s(u)}\big)^2}{P_{t,s}u}
 \le
P_{t,s}\left(\frac{\Gamma_s(u)}{u}\right).
\end{align*}

b) The dimension independent Harnack inequality  for  $\alpha_1$ and for $\alpha_2$
implies the  dimension independent Harnack inequality  for $\alpha_1\cdot\alpha_2$, \cite{wang2014analysis}, Thm.\ 1.4.2.
The dimension independent Harnack inequality  for a sequence  $\alpha_n\to\infty$ 
implies the  log-Harnack inequality.
In particular, the dimension independent Harnack inequality for  some $\alpha\in(1,\infty)$
implies the  dimension independent Harnack inequality  for all  $k\alpha, k\in\N$, and thus the log-Harnack inequality,  \cite{wang2014analysis}, Cor.\ 1.4.3.
\end{remark}

The {\bf proof} of the above theorem will be presented in the subsequent sections, decomposed into a variety of theorems devoted to individual implications.
In these theorems, we also specify  in detail the spaces of functions $u$ for which the respective inequalities are supposed to hold.
In Section 2 we prove the implications \eqref{est-III} $\Rightarrow$ \eqref{locpoinc1} $\Rightarrow$ \eqref{est-IV} and \eqref{est-III} $\Rightarrow$ \eqref{locpoinc2} $\Rightarrow$ \eqref{est-IV} as well as the implication 
\eqref{est-IV} $\Rightarrow$ \eqref{ptwBochner0}.
Section 3 is devoted to the proof of  the implications
\eqref{L1-grad} $\Rightarrow$ \eqref{logsob1} $\Rightarrow$ \eqref{ptwBochner0} and \eqref{L1-grad} $\Rightarrow$ \eqref{logsob2} $\Rightarrow$ \eqref{ptwBochner0}.
In Section 4 we prove  the implications
\eqref{L1-grad} $\Rightarrow$ \eqref{dimHarn} $\Rightarrow$ \eqref{logsob2} and in Section 5 the implication
\eqref{logHarn} $\Rightarrow$ \eqref{ptwBochner0}. This completes the proof of our theorem since 
\eqref{dimHarn} $\Rightarrow$ \eqref{logHarn} according to the previous remark,
\eqref{ptwBochner0} $\Rightarrow$ \eqref{L1-grad} according to Theorem \ref{eva},
and trivially \eqref{L1-grad} $\Rightarrow$ \eqref{est-III}.

\medskip

 The previous characterizations of super-Ricci flows easily extend to characterizations of $K$-super-Ricci flows for any $K\not=0$ 
by considering reparametrized mm-spaces $(X,\tilde d_t,\tilde m_t)_{t\in \tilde I}$ with $\tilde d_t=e^{-K\tau(t)}d_{\tau(t)}$, $\tilde m_t=m_{\tau(t)}$, and $\tilde I=\{t:\tau(t)\in I, 2Kt<C\}$ where $C\in \mathbb R$ and  $\tau(t)=-\frac1{2K}\log(C-2Kt)$, see Theorem 1.11 in \cite{sturm2016}.
Let us restrict ourselves to formulate this in the most simple case of static mm-spaces.


\setcounter{equation}{0}

\begin{cor}\label{static}
	 Let $(X,d,m)$ be a mm-space satisfying the {\rm RCD}$(-L,\infty)$ condition for some constant $L>0$. Then  the following assertions are equivalent:
	 
\begin{enumerate}
\item[\bf(i)] $(X,d,m)$ satisfies {\rm RCD}$(K,\infty)$.
\item[\bf(ii)] One/each of the local Poincar\'e inequalities holds
\begin{subequations}
\label{locpoincstatic}
 \begin{align*}
\textnormal{{{\bf (iia)}}} && P_{t}(u^2)(x)-(P_{t}u)^2(x)\leq&\frac{1-e^{-2Kt}}{K}P_{t}(\Gamma u)(x) \\
\textnormal{{{\bf (iib)}}} &&  P_{t}(u^2)(x)-(P_{t}u)^2(x)\geq &\frac{e^{2Kt}-1}{K}\Gamma(P_{t}u)(x). 
 \end{align*}
 \end{subequations}
\item[\bf(iii)] One/each of the local logarithmic Sobolev inequalities holds 
\begin{subequations}
\label{logsobstatic}
\begin{align*}
\textnormal{{{\bf (iiia)}}} && 
P_{t}(u\log u)-P_{t}u\log P_{t}u\leq& \frac{1-e^{-2Kt}}{2K}P_{t}\left(\frac{\Gamma(u)}{u}\right), \\
\textnormal{{{\bf (iiib)}}} &&
P_{t}(u\log u)-P_{t}u\log P_{t}u\geq& \frac{e^{2Kt}-1}{2K}\frac{\Gamma(P_{t}u)}{P_{t}u}.
\end{align*}
\end{subequations}
\item[\bf(iv)] The dimension independent Harnack inequality holds for one/each $\alpha\in (1,\infty)$
 \begin{align*}
  (P_{t}u)^\alpha(y)\leq P_{t}(u^\alpha)(x)\exp\left\{\frac{\alpha K d^2(x,y)}{2(\alpha-1)(1-e^{-2Kt})}\right\}.
 \end{align*}
 \item[\bf(v)] The 
 logarithmic Harnack inequality holds
 \begin{align*}
P_{t}(\log u)(x) \leq \log (P_{t}u)(y)+\frac{Kd^2(x,y)}{2(1-e^{-2Kt})}.
 \end{align*}
\end{enumerate}
\end{cor}

%
%

\begin{remark} 
So far, in the setting of mm-spaces only the implications {\bf(i)} $\Rightarrow$ {\bf(iib)}, {\bf(i)} $\Rightarrow$ {\bf(iiib)}
{\bf(i)} $\Rightarrow$ {\bf(v)}, and {\bf(i)} $\Rightarrow$  {\bf(iv)} 
were known (Thm.~6.8 in \cite{agmr}, Cor. 4.4 in \cite{tamanini2019harnack}, Lemma 4.6 in \cite{agsbe}, and Thm.~3.1 in \cite{H.Li}). The implications {\bf(i)} $\Rightarrow$ {\bf(iia)} and {\bf(i)} $\Rightarrow$ {\bf(iiib)} are new also in the static case. In particular, none of the reverse implications {\bf(iia)} $\Rightarrow$ {\bf(i), (iib)} $\Rightarrow$ {\bf(i), (ii), (iii), (iv)}, or {\bf(v)} $\Rightarrow$ {\bf(i)} was proven before for mm-spaces.
  
  Also so far, for the implication {\bf(v)} $\Rightarrow$ {\bf(i)} no proof exists in the setting of $\Gamma$-calculus for diffusion semigroups.
\end{remark}
\subsection{Preliminaries}

Let us recall some basic properties of the heat propagators $P_{t,s}$ and their adjoints $P^*_{t,s}$, see Section 3 in \cite{sturm2016}.
We call $u$ a solution to the heat equation on $(s,\tau)\times X$ if $u\in L^2((s,\tau); \F)\cap H^1((s,\tau); \F^*)$ and 
\begin{equation}\label{heat-def2}-\int_s^\tau\E_r(u_r,w_r)dr=\int_s^\tau  \langle\partial_ru_r ,w_r \rangle \, dr\end{equation}
for all $w\in L^2((s,\tau); \F)$, where $\langle\cdot,\cdot\rangle$ denotes the dual pairing between $\F$ and $\F^*$.
 Note that the solution $u$ lies in $\mathcal C([s,\tau]; L^2(X))$ so that the values at $t=s$ and $t=\tau$ exist.  For all $h\in L^2(X)$ there exists a unique solution $u_t=P_{t,s}h$ with $u_s=h$.

We call $v$ a solution to the adjoint heat equation on $(\sigma,t)\times X$ if $v\in L^2((\sigma,t); \F)\cap H^1((\sigma,t); \F^*)$ and
\begin{align}\label{ad-heat-def2}
\int_\sigma^t \E_s(v_s,w_s)\, ds + \int_\sigma^t \int v_s w_s \partial_s f_s \, dm_s\, ds=
\int_\sigma^t  \langle\partial _s v_s, w_s\rangle \, ds
\end{align}
for all $w\in L^2((s,\tau); \F)$. Again the solution $v$ lies in $\cC([\sigma,t]; L^2(X,m))$. 
For each $g\in L^2(X)$ there exists a unique solution $v_s=P_{t,s}^*g$ with $v_t=g$.

The relation between the heat flow and its adjoint is given by
\begin{align}\label{eq: adjoint}
\int P_{t,s} h\, g \, dm_t =&\int h\, P^*_{t,s}g\, dm_s,\qquad
\hat P_{t,s}(g\, m_t)=(P_{t,s}^*g)\, m_s.
\end{align}

We further collect the following properties from \cite{sturm2016}.
\begin{lma}[\cite{sturm2016}, Prop. 2.14]\label{lma: maximum} For all $u\in L^2(X,m)$ and all $s<t$, $p\in[1,\infty)$
\begin{enumerate}
\item $u\ge0\Longrightarrow P_{t,s}u\ge0$,\qquad $u\le M\Longrightarrow P_{t,s}u\le M$.
\item $v\ge0\Longrightarrow P^*_{t,s}v\ge0$,\qquad $v\le M\Longrightarrow P^*_{t,s}v\le Me^{L(t-s)}$.
\item $\|P_{t,s}u\|_{L^p(m_t)}\le e^{L(t-s)/p}\cdot \|u\|_{L^p(m_s)}$, \qquad
 $\|P^*_{t,s}v\|_{L^p(m_s)}\le e^{L(t-s)(1-1/p)}\cdot \|v\|_{L^p(m_t)}$.
\end{enumerate}
\end{lma}
These estimates allow to extend the propagators $P_{t,s}$ and their adjoints $P^*_{t,s}$ in the canonical way from operators on $L^2(X,m)$ to operators on $L^p(X,m)$ for any $p\in[1,\infty]$.

\begin{Prop}[\cite{sturm2016}, Theorem 2.12] \label{prop:prop}
The following properties hold.
\begin{enumerate}
\item  Let $u_t=P_{t,s}u$. Then $u_t\in\D(\Delta_t)$ for a.e. $t>s$ and if $u_s\in \F$
\begin{align*}
\int_s^\tau\int |\Delta_t u_t|^2\, dm_t\, d t\leq C(\E_s(u_s)-\E_\tau(u_\tau)),
\end{align*}
where $s<\tau<T$ and $C>0$ only depends on the Lipschitz constants of $t\mapsto f_t$ and $t\mapsto\log d_t$. Moreover
\begin{align*}
\lim_{h\to0}\frac1h(u_{t+h}-u_t)=\Delta_tu_t
\end{align*}
in $L^2(X)$ for a.e. $t>s$.
\item Let $v_s=P^*_{t,s}v$. Then $v_s\in\D(\Delta_s)$ for a.e. $s<t$ and if $v_t\in \F$
\begin{align*}
\int_\sigma^t\int |\Delta_s v_s|^2\, dm_s\, ds\leq C(\E_t(v_t)-\E_\sigma(v_\sigma))+C\int_\sigma^t\int|v_s|^2\, dm_s\, ds,
\end{align*}
where $0<\sigma<t$ and $C>0$ only depends on the Lipschitz constants of $t\mapsto f_t$ and $t\mapsto\log d_t$. Moreover
\begin{align*}
\lim_{h\to0}\frac1h(v_{s+h}-v_s)=-\Delta v_s+v_s\dot f_s
\end{align*}
in $L^2(X)$ for a.e. $s<t$.
\end{enumerate}
\end{Prop}

\section{The local and the reverse local Poincar\'e inequalities}

For later purposes it will be convenient to present the notion of semigroup mollification introduced in \cite[Sec. 2.1]{agsbe}.
\begin{Def}\label{molli}
Let $t\in (0,T)$ and $\kappa\in\cC_c^\infty(0,\infty)$ with $\kappa\geq0$ and $\int_0^\infty\kappa(r)\, dr=1$. Let $(H^t_r)_{r\ge0}$ denote the heat semigroup in the static mm-space $(X,d_t,m_t)$. For $\eps>0$ and $\psi\in \F\cap L^\infty(X)$ we define
\begin{align*}
 \psi_\eps=\frac1\varepsilon\int_0^\infty H^t_r\psi\, \kappa(r/\varepsilon)\, dr.
\end{align*}
 \end{Def}
 It is immediate to verify that $\psi_\eps,\Delta_t \psi_\eps\in\D(\Delta_t)\cap\Lip_b(X)$ and $\psi_\eps\to\psi$ in $\F$ as $\eps\to0$, see e.g. \cite[Sec 2.1]{agsbe}.

\subsection{From $L^2$-gradient estimate to local  and reverse local Poincar\'e inequalities}

\begin{thm}\label{thmlocpoinc}
Suppose that for all $u\in \F$ and all $s<t$ the $L^2$-gradient estimate 
\begin{align}\label{l2gradest}
 \Gamma_t( P_{t,s}u)\leq P_{t,s}(\Gamma_s u)\qquad m \mbox{-a.e. on }X
\end{align}
holds. 
Then we have for all $u\in\F$, $s<t$
 \begin{align}\label{upper local Poincare}
  P_{t,s}(u^2)-(P_{t,s}u)^2\leq&2(t-s)P_{t,s}(\Gamma_su)\qquad m \mbox{-a.e. on }X
  \end{align}
  and for all $u\in L^2(X)$, $s<t$
   \begin{align}\label{lower local Poincare}
  P_{t,s}(u^2)-(P_{t,s}u)^2\geq &2(t-s)\Gamma_t(P_{t,s}u)\qquad m \mbox{-a.e. on }X.
 \end{align}
In particular, for $u\in L^2(X)\cap L^\infty(X)$
\begin{align}\label{uniform grad est}
 \Gamma_t(P_{t,s}u)\leq \frac{||u||^2_\infty}{2(t-s)}.
\end{align}
\end{thm}

\begin{proof}
Let $u=u_s$ and $g=g_t$ be both elements in $\F\cap L^\infty(X)$ and consider on $(s,t)\times X$ the solutions to the heat equation and adjoint heat equation
$$u_r:=P_{r,s}u_s, \quad g_r=P^*_{t,r}g_t.$$

Due to Proposition \ref{prop:prop} we have $u_r,g_r\in H^1((s,t);L^2(X))$.
Since $u_r,g_r\in H^1((s,t); L^2(X))$ and $e^{-f_r}\in \Lip((s,t);L^\infty(X))$ we deduce that the function $r\mapsto \int u_r^2 g_r\, dm_r$ is locally absolutely continuous.
The almost everywhere derivative can be computed as
\begin{align*}
\frac{d}{dr}\int u_r^2g_rdm_r=&\lim_{h\to0}\int\frac{(g_{r+h}-g_r)}h u_{r+h}^2\, dm_{r+h}
+\lim_{h\to0} \int g_r\frac{(u^2_{r+h}-u^2_r)}{h}\, dm_{r+h}\\
+&\lim_{h\to0} \int g_ru_r^2 \frac{(e^{-f_{r+h}}-e^{-f_r})}h\, dm\\
=&\int\partial_rg_r u_{r}^2\, dm_{r}
+\int g_r2u_r\partial u_r\, dm_{r}
- \int g_ru_r^2 \partial_rf_r\, dm_r,\\
\end{align*}
where the last equality holds since $\frac{g_{r+h}-g_r}h\to \partial_rg_r,\frac{u_{r+h}-u_r}h\to \partial_ru_r$ in $L^2(X)$ for almost every $r$ and since the mapping $z\mapsto z^2\in \cC^2(\R)$.

Then, by the defining properties of the heat equation \eqref{heat-def2}, \eqref{ad-heat-def2}
\begin{eqnarray*}
-2\int_s^t\int g_r\Gamma_r(u_r)dm_rdr&=&
\int_s^t\int -2\Gamma_r(g_ru_r,u_r)+\Gamma_r(u_r^2,g_r)dm_rdr\\
&=&\int_s^t\int \big(2g_ru_r\partial_ru_r+u_r^2\partial_rg_r-u_r^2g_r\partial_rf_r\big)dm_rdr\\
&=&\int_s^t\frac{d}{dr}\Big(\int u_r^2g_rdm_r\Big)dr=\int u_t^2g_tdm_t-\int u_s^2g_sdm_s.
\end{eqnarray*}
This proves 
\begin{align}\label{erste-abl}
\int g((P_{t,s}u)^2-P_{t,s}(u^2))\, dm_t
=-&2\int_s^t\int P^*_{t,r} g(\Gamma_r(P_{r,s}u))\, dm_r\,dr.
\end{align}
Applying \eqref{l2gradest} to $\Gamma_r(P_{r,s}u)$ on the right hand side gives
\begin{align*}
\int g((P_{t,s}u)^2-P_{t,s}(u^2))\, dm_t
\geq-&2(t-s)\int gP_{t,s}(\Gamma_s(u))\, dm_t,
\end{align*}
and applying \eqref{l2gradest} to $ P_{t,r}\Gamma_r$ gives
\begin{align*}
\int g((P_{t,s}u)^2-P_{t,s}(u^2))\, dm_t
\leq
-&2(t-s)\int g\Gamma_t(P_{t,s}(u))\, dm_t.
\end{align*}
Since $g$ is arbitrary, this proves the first two claims of the theorem in the case of bounded $u\in\F$. 
The claim  \eqref{lower local Poincare} for bounded $u\in L^2(X)$ follows by applying the latter estimate with $s+\delta$ in the place of $s$ to the function $P_{s+\delta,s}u$ as $\delta\to 0$, which lies in $\F$ and from $\int g P_{t,s+\delta}((P_{s+\delta,s}u)^2)dm_t\to \int gP_{t,s}(u^2)dm_t$ which in turn is a consequence of the continuity of $\delta\mapsto P_{t,s+\delta}^*g$ and of $\delta\mapsto P_{s+\delta,s}u$ in $L^2$ and the uniform boundedness of the latter in $L^\infty$. 

Thanks to the  monotonicity (w.r.t. $C\mapsto u\wedge C$ or $C\mapsto u\vee -C$) of all the involved quantities,
the claims for unbounded $u$ will follow by a simple truncation argument. Indeed,
$u\wedge C\vee -C\to u$ in $L^2$ and thus, since $g$ is bounded, $\int g(P_{t,s}u\wedge C\vee -C)^2dm_t\to \int g(P_{t,s}u)^2dm_t$ as well as $\int(u\wedge C\vee -C)^2 P_{t,s}^*g\, dm_s\to \int  u^2P_{t,s}^*g\,dm_s$. Moreover, under the heat flow the initial $L^2$-convergence will be improved to a $\F$-convergence. Thus 
$$\int g\Gamma_t(P_{t,s}(u\wedge C\vee -C))\, dm_t\to \int g\Gamma_t(P_{t,s}(u))\, dm_t.$$
Finally, for the remaining term it suffices to observe that 
$$\int gP_{t,s}(\Gamma_t(u\wedge C\vee -C))\, dm_t\le \int gP_{t,s}(\Gamma_t(u))\, dm_t.$$
\end{proof}

\subsection{From reverse local Poincar\'e inequality to dynamic Bochner inequality}

\begin{thm}\label{thm poinc}
Suppose that  the reverse local Poincar\'e inequality  holds:  \\
for all $s<t$ and for all $u\in \F\cap L^\infty(X)$
\begin{align*}
 P_{t,s}(u^2)-(P_{t,s}u)^2\geq&2(t-s)\Gamma_t (P_{t,s}u)\qquad m \mbox{-a.e. on }X.
 \end{align*} 
 Then the 
 dynamic Bochner inequality \eqref{est-IV}  holds true (`integrated in time'): \\
 $\forall S,T\in I$, $\forall u,g\in\F$ with $g\in L^\infty(X)$, $u\in \Lip(X)$ and for a.e. $q\in (S,T)$ 
 $$\int \Big[(\Delta_q g_q) \Gamma_q(u_q)
+2(\Delta_q u_q)^2g_q
+ 2\Gamma_q(u_q,g_q)\Delta_q u_q
-\stackrel{\bullet}{\Gamma}_q(u_q)g_q\Big] dm_q\ge 0$$
where $u_q:=P_{q,S}u, g_q=P^*_{T,q}g$.
\end{thm}

\begin{proof}
Given $u\in \F\cap L^\infty(X)$ and nonnegative 
 $g\in L^1(X)\cap L^\infty(X)$  we have shown in  \eqref{erste-abl} 
 that for all $s<t$
\begin{align*}
\int g(P_{t,s}(u^2)-(P_{t,s}u)^2)\, dm_t
=&2\int_s^t\int P^*_{t,r}g\Gamma_r(P_{r,s}u)\, dm_rdr.
\end{align*}
Approximation by truncated $u$'s easily allows to extend the assertion to all $u\in \F$.
The local Poincar\'e inequality, therefore, implies
\begin{eqnarray*}
0&\le&
\frac1{(t-s)^2}\int\Big[
P_{t,s}u^2-(P_{t,s}u)^2 -2(t-s)\Gamma_t(P_{t,s}u)\Big]g\, dm_t\\
&=&
\frac2{(t-s)^2}\int_s^t\int g\Big[P_{t,r}\Gamma_r(P_{r,s}u)-\Gamma_t(P_{t,s}u)\Big]\, dm_tdr.
\end{eqnarray*}
Now let us fix $S,T\in I$ and choose $g_T,u_S\in\F$ with $g_T\in L^\infty$ and $u_S\in \Lip(X)$. Given $s,t$ with $S<s<t<T$, we put
$$g_t=P^*_{T,t}g_T, \quad 
u_s=P_{s,S}u_S$$ and apply the previous estimate with $g_t, u_s$ in the place of $g,u$. Then 
\begin{eqnarray*}
0\le
\frac2{(t-s)^2}\int_s^t\int g_t\Big[P_{t,r}\Gamma_r(u_r)-\Gamma_t(u_t)\Big]\, dm_tdr
= \frac2{(t-s)^2}\int_s^t[\Psi(r)-\Psi(t)]dr
\end{eqnarray*}
where we defined
\begin{align}\label{eq: Psi}
\Psi(q)=\int g_q\Gamma_q(u_q)\, dm_q.
\end{align}
Following the proof of Theorem 5.7 in \cite{sturm2016} we have
\begin{align*}
\Psi(r)-\Psi(t)\le\int_r^t\int \Big[(\Delta_q g_q) \Gamma_q(u_q)
+2(\Delta_q u_q)^2g_q
+ 2\Gamma_q(u_q,g_q)\Delta_q u_q
-\stackrel{\bullet}\Gamma_q(u_q)g_q\Big] dm_q
dq
\end{align*}
and hence
\begin{align*}
0\le&\frac1{(t-s)^2}\int_s^t\int\Big[P_{t,s}u_s^2-(P_{t,s}u_s)^2 -2(t-s)\Gamma_t(P_{t,s}u_s)\Big]g_t\, dm_t\,dr\\
\leq
&\frac2{(t-s)^2}\int_s^t
\int_r^t\int \Big[(\Delta_q g_q) \Gamma_q(u_q)
+2(\Delta_q u_q)^2g_q
+ 2\Gamma_q(u_q,g_q)\Delta_q u_q
-\stackrel{\bullet}{\Gamma}_q(u_q)g_q\Big] dm_q
dq
dr\\
=
&\frac2{(t-s)^2}\int_s^t(q-s)
\int \Big[(\Delta_q g_q) \Gamma_q(u_q)
+2(\Delta_q u_q)^2g_q
+ 2\Gamma_q(u_q,g_q)\Delta_q u_q
-\stackrel{\bullet}{\Gamma}_q(u_q)g_q\Big] dm_q
dq.
\end{align*}
Since this holds for all $(s,t)\subset (S,T)$, it implies (by Lebesgue's density theorem) that 
$$0\le \int \Big[(\Delta_q g_q) \Gamma_q(u_q)
+2(\Delta_q u_q)^2g_q
+ 2\Gamma_q(u_q,g_q)\Delta_q u_q
-\stackrel{\bullet}{\Gamma}_q(u_q)g_q\Big] dm_q$$
 for a.e. $q\in (S,T)$. This is the claim, namely the dynamic Bochner inequality  \eqref{est-IV}.
\end{proof}

\subsection{From local Poincar\'e inequality  to dynamic Bochner inequality }

For the proof of the following implication, we will make the additional a priori assumption that 
 \begin{equation}\label{a priori grad}\sup_{t} \|\Gamma_t( P_{t,s}u)\|_\infty<\infty\tag{{\bf A2.a}}
 \end{equation}
   for each $u\in \Lip(X)$. Note that this assumption is always fullfilled in the time-independent case
   thanks to the  RCD$(K,\infty)$-condition as one of our standing assumptions. 
   
\begin{thm}
Suppose \eqref{a priori grad} and that  the local Poincar\'e inequality holds:\\
  for all $s<t$ and for all $ u\in \F\cap L^\infty(X)$
\begin{align*}
  P_{t,s}(u^2)-(P_{t,s}u)^2\leq &2(t-s)P_{t,s}(\Gamma_su)\qquad m \mbox{-a.e. on }X.
 \end{align*} 
 Then the dynamic Bochner inequality \eqref{est-IV} holds true (`integrated in time'). 
\end{thm}

\begin{proof} The proof is very similar to that of the previous theorem.
Now the a priori assumption is required to guarantee appropriate integrability of the involved quantities (which in the previous case was a simple consequence of the assumption, cf. estimate \eqref{uniform grad est}). 

Then, as in the proof of Theorem \ref{thm poinc}, the local Poincar\'e inequality implies 
\begin{eqnarray*}
0&\ge&
\frac1{(t-s)^2}\int g\Big[
P_{t,s}u^2-(P_{t,s}u)^2 -2P_{t,s}\Gamma_s(u)\Big]\, dm_t\\
&=&\frac2{(t-s)^2}\int_s^t\int g\Big[
P_{t,r}\Gamma_r(u_r)-P_{t,s}\Gamma_s(u)\Big]\, dm_t\, dr
=\frac2{(t-s)^2}\int_s^t\Big[\Psi(r)-\Psi(s)\Big]\, dr,
\end{eqnarray*}
where $\Psi$ is defined in \eqref{eq: Psi}. Consequently, arguing as in the proof of Theorem \ref{thm poinc},
\begin{eqnarray*}
0&\geq&-\frac2{(t-s)^2}\int_s^t
\int_s^r\int \Big[(\Delta_q g_q) \Gamma_q(u_q)
+2(\Delta_q u_q)^2g_q
+ 2\Gamma_q(u_q,g_q)\Delta_q u_q)
-\stackrel{\bullet}{\Gamma}_q(u_q)g_q\Big] dm_q
dq
dr\\
&=&-\frac2{(t-s)^2}\int_s^t(t-q)
\int \Big[(\Delta_q g_q) \Gamma_q(u_q)
+2(\Delta_q u_q)^2g_q
+ 2\Gamma_q(u_q,g_q)\Delta_q u_q)
-\stackrel{\bullet}{\Gamma}_q(u_q)g_q\Big] dm_q
dq.
\end{eqnarray*}
Again by Lebesgue's density theorem this implies that 
$$0\le \int \Big[(\Delta_q g_q) \Gamma_q(u_q)
+2(\Delta_q u_q)^2g_q
+ 2\Gamma_q(u_q,g_q)\Delta_q u_q)
-\stackrel{\bullet}{\Gamma}_q(u_q)g_q\Big] dm_q$$
 for a.e. $q\in (S,T)$. 
\end{proof}

\subsection{From dynamic Bochner inequality (`integrated in time') to dynamic Bochner inequality pointwise in time}

In addition to our standing assumptions, let us now assume that 
\begin{itemize}
 \item   the domains $\D(\Delta_t)$ are independent of $t\in (0,T)$ and for $u,g\in\D(\Delta)$ with  $\Delta_t u, \Delta_tg\in L^\infty(X)$ the functions
\begin{equation}\label{Laplace cont}\tag{{\bf A2.b}}
 \begin{aligned}
 r\mapsto \Delta_r u,\quad
 q\mapsto \Delta_r P_{q,s}u,\quad
  q\mapsto \Delta_qP_{q,s}u, \quad q\mapsto \Delta_qP^*_{t,q}g, 
 \end{aligned}
\end{equation}
are continuous in $L^2(X)$ and bounded in $L^\infty(X)$;
\item for $u\in\F$ the function $\partial_s\Gamma_s(u)$ exists in $L^1(X)$ and the map
\begin{align}\label{eq: cont derivative}\tag{{\bf A2.c}}
I\times\F \ni (s,u)\mapsto \partial_s\Gamma_s(u)
\end{align}
is continuous  in $L^1(X)$.
\end{itemize}

Note that all these assumptions are trivially satisfied in the static case.

\begin{lma}\label{lma: con delta}
  The assumption \eqref{Laplace cont} implies that for $u,g\in\D(\Delta)$ with  $\Delta_t u, \Delta_tg\in L^\infty(X)$ the functions 
\begin{align*}
q\mapsto P_{q,s}u, \quad q\mapsto P^*_{t,q}g
\end{align*}
are continuous in $\F$. 
\end{lma}

\begin{proof}
This follows from integration by parts.
\end{proof}

\begin{thm}
Under the previous assumptions, the dynamic Bochner inequality \eqref{est-IV} implies the following `dynamic  Bochner inequality pointwise in time':\\
$\forall t\in I$, 
$\forall u,g\in\D(\Delta)\cap L^\infty(X)$ with  $\Gamma_t(u) \in L^\infty(X)$ and $g\ge0$
\begin{equation}\label{ptwBochner}\int \Big[(\Delta_t g) \Gamma_t(u)
+2(\Delta_t u)^2g
+ 2\Gamma_t(u,g)\Delta_t u)
-\partial_t \Gamma_t(u)g\Big] dm_t\ge0.
\end{equation}
\end{thm}

\begin{proof} Given $t\in I$, 
$u,g\in\D(\Delta)\cap L^\infty(X)$ with  $\Gamma_t(u), \Delta_tu, \Delta_t g \in L^\infty(X)$ and $g\ge0$, choose $s<t$ and define 
 $u_{q,s}:=P_{q,s}u, g_q=P^*_{t,q}g$ for $q\in [s,t]$. Then the dynamic Bochner inequality in its integrated version and \eqref{eq: cont derivative} imply that the function 
 $$q\mapsto \int \Big[(\Delta_q g_q) \Gamma_q(u_{q,s})
+2(\Delta_q u_{q,s})^2g_q
+ 2\Gamma_q(u_{q,s},g_q)\Delta_q u_{q,s}
-\partial_q\Gamma_q(u_q)g_q\Big] dm_q$$ is nonnegative for a.e. $q$. Moreover, according to \eqref{Laplace cont}, Lemma \ref{lma: con delta} and \eqref{eq: cont derivative}, this function is continuous. Thus, in particular, it is nonnegative for $q=s$, i.e.
$$\int \Big[(\Delta_s P^*_{t,s}g) \Gamma_s(u)
+2(\Delta_s u)^2P^*_{t,s}g
+ 2\Gamma_s(u,P^*_{t,s}g)\Delta_s u
-\partial_s\Gamma_s(u)P^*_{t,s}g\Big] dm_s\ge0.$$
Now finally we consider the limit $s\to t$ which implies $P^*_{t,s}g\to g$ in $L^2(X)$ as well as
$\Delta_sP^*_{t,s}g\to \Delta_tg$ by \eqref{Laplace cont}. According to Lemma \ref{lma: con delta}, 
$P^*_{t,s}g\to g$ in $\F$. 
Therefore,
$$\int \Big[(\Delta_t g) \Gamma_t(u)
+2(\Delta_t u)^2g
+ 2\Gamma_t(u,g)\Delta_t u
-\partial_t\Gamma_t(u)g\Big] dm_t\ge0.$$
To obtain the estimate for general $u,g$, we approximate them using the static $(X,d_t,m_t)$-heat semigroup mollifier from Definition \ref{molli}.
\end{proof}

\section{The local logarithmic Sobolev inequalities}

\subsection{From $L^1$-gradient estimate to local logarithmic Sobolev  inequality}
\begin{thm}\label{thmlocsob}
Suppose that the $L^1$-gradient estimate 
\begin{align}\label{l1gradest}
\sqrt{ \Gamma_t (P_{t,s}u)}
\leq 
P_{t,s}\sqrt{\Gamma_s (u)}
\end{align}
  holds for every $s<t$, $u\in\F$ and $m$-a.e.. Then for every $s<t$ and $u\geq 0$ such that $u\in\D(S)$ and $\sqrt u\in \F$ we have $m$-a.e. on $X$
\begin{align}\label{eq:upper local log}
P_{t,s}(u\log u)-P_{t,s}u\log P_{t,s}u\leq& (t-s)P_{t,s}\left(\frac{\Gamma_s(u)}{u}\right)\\
\label{eq:lower local log}
P_{t,s}(u\log u)-P_{t,s}u\log P_{t,s}u\geq& (t-s)\frac{\Gamma_t(P_{t,s}u)}{P_{t,s}u}.
\end{align}
Estimate \eqref{eq:lower local log} holds more generally for all nonnegative $u\in\D(S)\cap L^1(X)$.
\end{thm}

\begin{proof}
Define for $s<r<t$, $g\in L^1(X)\cap L^\infty(X)\cap \F$ such that $g\geq 0$ and $u\in\D(S)\cap L^\infty(X)$ such that $M\geq u\geq 0$ for some constant $M$ and $\sqrt u\in \F$
\begin{align*}
\Psi_\varepsilon(r):=\int g_r\, \psi_\varepsilon(u_r)\, dm_r,
\end{align*}
where $g_r=P_{t,r}^*  g$ and $u_r=P_{r,s}u$ and 
where $\psi_\varepsilon(z)\colon [0,\infty)\to\mathbb R$ by setting $\psi_\varepsilon'(z)=\log(z+\varepsilon)+1$ and $\psi_\varepsilon(0)=0$. 

Since $g_t,u_s\in \F$ we have by virtue of Proposition \ref{prop:prop} that $g_r,u_r\in H^1((s,t);L^2(X))$.
Since $g_r,u_r\in H^1((s,t);L^2(X))$ and $e^{-f_r}\in \Lip((s,t);L^\infty(X))$, we deduce that the map $r\mapsto \Psi_\eps(r)$ is locally absolutely continuous.
 Then, since $\psi_\eps'\in \Lip_b([0,M])$, we compute similarly as in the proof of Theorem \ref{thmlocpoinc} 
 
\begin{align*}
\frac{d}{dr}\Psi_\varepsilon(r)=&\int \Gamma_r(g_r,u_r)\psi_\varepsilon'(u_r)-\Gamma_r(g_r\psi_\varepsilon'(u_r),u_r)\, dm_r\\
=-&\int g_r\, \psi_\varepsilon''(u_r)\Gamma_r(u_r)\, dm_r
=-\int g_r\frac{\Gamma_r(u_r)}{u_r+\varepsilon}\, dm_r.
\end{align*}
Using the Cauchy-Schwarz inequality and \eqref{l1gradest} 
we find for the integrand

\begin{align*}
&P_{t,r}\left(\frac{\Gamma_r(u_r)}{u_r+\varepsilon}\right)\leq P_{t,r}\left(\frac{(P_{r,s}\sqrt{\Gamma_s(u)})^2}{u_r+\varepsilon}\right)\\
=&P_{t,r}\left(\frac{\Big(P_{r,s}\Big(\frac{\sqrt{\Gamma_s(u)}}{\sqrt{u+\varepsilon}}\sqrt{u+\varepsilon}\Big)\Big)^2}{u_r+\varepsilon}\right)
\leq P_{t,r}\left(\frac{P_{r,s}\Big(\frac{\Gamma_s(u)}{u+\varepsilon}\Big)(u_r+\varepsilon)}{u_r+\varepsilon}\right)\\
=&P_{t,s}\left(\frac{\Gamma_s(u)}{u+\varepsilon}\right).
\end{align*}
Integration over $(s,t)$ yields
\begin{align}\label{smoothest}
 \int g\psi_\varepsilon(P_{t,s}u)\, dm_t-\int gP_{t,s}(\psi_\varepsilon(u))\, dm_t\geq-(t-s)\int gP_{t,s}\left(\frac{\Gamma_s(u)}{u+\varepsilon}\right)\, dm_t.
\end{align}

Since $u\in\D(S)$ we have by Proposition 2.8 in \cite{sturm2016} that $P_{t,s}u\in \D(S)$ and we find by dominated convergence that the left hand side converges as $\eps\to0$ to
\begin{align*}
 \int gP_{t,s}u\log(P_{t,s}u)\, dm_t-\int gP_{t,s}(u\log u)\, dm_t,
\end{align*}
while by monotone convergence the right hand side converges to
\begin{align*}
 -(t-s)\int gP_{t,s}\left(\frac{\Gamma_s(u)}{u}\right)\, dm_t,
\end{align*}
and hence 
\begin{align}\label{onesidedest}
  \int gP_{t,s}u\log(P_{t,s}u)\, dm_t-\int gP_{t,s}(u\log u)\, dm_t\geq-(t-s)\int gP_{t,s}\left(\frac{\Gamma_s(u)}{u}\right)\, dm_t.
\end{align}

By taking $u^n:=u\wedge n$ and letting $n\to\infty$ we obtain
\eqref{onesidedest} for general $u\in\D(S)$ with $\sqrt u\in \F$, since $u^n\to u$ and $P_{t,s}u^n\to P_{t,s}u$ in $L^1(X)$, and $\Gamma(u^n)=\Gamma(u)1_{\{u<n\}}$ a.e..

Since $g$ is arbitrary we find for a.e. $x\in X$
\begin{align*}
P_{t,s}(u\log u)-P_{t,s}u\log P_{t,s}u\leq (t-s)P_{t,s}\left(\frac{\Gamma_s(u)}{u}\right).
\end{align*}

To obtain the reverse bound \eqref{eq:lower local log} we apply Jensen's inequality to the functions $\eta(z)=z^2$ and $\beta(z,w)=z^2/w$, which amounts to

\begin{align*}
P_{t,r}\left(\frac{\Gamma_r(P_{r,s}u)}{P_{r,s}u}\right)\geq \frac{P_{t,r}\Gamma_r(P_{r,s}u)}{P_{t,s}u}
\geq  \frac{(P_{t,r}\sqrt{\Gamma_r(P_{r,s}u)})^2}{P_{t,s}u}\geq  \frac{\Gamma_t(P_{t,s}u)}{P_{t,s}u}.
\end{align*}

A similar argumentation as above yields the desired estimate.
\end{proof}

\subsection{From local logarithmic Sobolev inequalities to dynamic Bochner inequality}
For this subsection we will additionally assume that (\textbf{A2.a-c})
hold. Moreover, we assume that $m_t(X)<\infty$ for some (hence all) $t\in(0,T)$ and that 
\begin{itemize}
\item for all fixed $s\in (0,T)$ and all $u\in \D(\Delta)\cap L^\infty(X)$ such that $\Delta_s u\in L^\infty(X)$
\begin{align}\label{eq: cont bounded}\tag{{\bf A3}}
 q\mapsto P_{q,s}u \text{ is continuous in }L^\infty(X);
\end{align}
\end{itemize}

Note that (\textbf{A3}) is always satisfied for the usual heat flow $(P_t)_{t\geq0}$ on RCD$(K,\infty)$-spaces, see Lemma \ref{regofexp}.

We show the following.
\begin{thm}
Assume that one of the local logarithmic Sobolev inequalities, \eqref{eq:upper local log} or \eqref{eq:lower local log}, holds. Then the pointwise dynamic Bochner holds for $t$, i.e. for all $v\in\D(\Delta_t)\cap L^\infty(X)$ such that $\Gamma_t(v)\in L^\infty(X)$ and all $g\in\D(\Delta_t)\cap L^\infty(X)$ with $g\geq0$ holds
\begin{align}\label{conflicting bochner}
 \frac12\int\Gamma_t(v)\Delta_t g\, dm_t+\int(\Delta_t v)^2g+\Gamma_t(v,g)\Delta_t v\, dm_t\geq\frac12 \int (\partial_t \Gamma_t)(v) g\, dm_t.
\end{align}
\end{thm}

\begin{proof}
Let $v,\Delta_t v\in\D(\Delta_t)\cap\Lip_b(X)$. Define $u=e^v$. Then
$u\in \D(S)\cap\Lip_b(X)\cap \D(\Delta_t)$ with $\Delta_t u\in L^\infty(X)\cap \F$ and there exists constants $M,c$ such that $M\geq u\geq c>0$.
Let $g\in \D(\Delta_t)\cap L^\infty(X)$ with $g\geq0$. 

Then we know  from the proof of Theorem \ref{thmlocsob} that 
\begin{align*}
\int g(P_{t,s}(u\log u)-P_{t,s}(u)\log P_{t,s}(u))\, dm_t
=\int_s^t\int g_r\frac{\Gamma_r(u_r)}{u_r}\, dm_r\, dr.
\end{align*}

Together with \eqref{eq:upper local log} we find
\begin{equation}
\begin{aligned}\label{eq: ineq}
0\geq&\int g\left(P_{t,s}(u\log u)-P_{t,s}(u)\log P_{t,s}(u)-(t-s)P_{t,s}\Big(\frac{\Gamma_s(u)}{u}\Big)\right)\, dm_t\\
=&\int_s^t\int g_r\frac{\Gamma_r(u_r)}{u_r}\, dm_r-\int gP_{t,s}\left(\frac{\Gamma_s(u)}{u}\right)\, dm_t\, dr,
\end{aligned}
\end{equation}
where $u_r=P_{r,s}u$ and $g_r=P_{t,r}^*g$.

We now claim that the map 
\begin{align*}
r\mapsto \int g_r\frac{\Gamma_r( u_r)}{u_r}\, dm_r
\end{align*}
is absolutely continuous. To this end we compute for a.e. $r_1,r_2\in(s,t)$ with $r_1<r_2$
\begin{align*}
\left| \int g_{r_2}\frac{\Gamma_{r_2}( u_{r_2})}{u_{r_2}}\, dm_{r_2}- \int g_{r_1}\frac{\Gamma_{r_1}( u_{r_1})}{u_{r_1}}\, dm_{r_1}\right|
\leq &\left|\int_{r_1}^{r_2}\int \Delta_r g_r\frac{\Gamma_{r_2}( u_{r_2})}{u_{r_2}}\, dm_r\, dr\right|\\
&+\frac1{c^2}\left|\int_{r_1}^{r_2}\int  g_{r_1}\Delta_ru_{r}\Gamma_{r_2}(u_{r_2})\, dm_{r_1}\, dr\right|\\
&+C(r_2-r_1)\left|\int  g_{r_1}\frac{\Gamma_{r_1}( u_{r_2})}{u_{r_1}}\, dm_{r_1}\right|\\
&+\left|\int  \frac{g_{r_1}}{u_{r_1}}(\Gamma_{r_1}(u_{r_2})-\Gamma_{r_1}(u_{r_1}))\, dm_{r_1}\right|,
\end{align*}
where we applied \eqref{eq: s.f.} for the third term on the right hand side.

The first three terms are finite by virtue of \eqref{a priori grad} and Proposition \ref{prop:prop}. For the last one we further compute
\begin{align*}
\left|\int  \frac{g_{r_1}}{u_{r_1}}(\Gamma_{r_1}( u_{r_2})-\Gamma_{r_1}( u_{r_1}))\, dm_{r_1}\right|
=&\left|\int  \frac{g_{r_1}}{u_{r_1}}\Gamma_{r_1}( u_{r_2}- u_{r_1}, u_{r_2}+ u_{r_1})\, dm_{r_1}\right|\\
\leq&\left|\int \frac{g_{r_1}}{u_{r_1}}( u_{r_2}- u_{r_1})\Delta_{r_1}( u_{r_2}+ u_{r_1})\, dm_{r_1}\right|\\
&+\left|\int \Gamma_{r_1}\left( u_{r_2}+ u_{r_1},\frac{g_{r_1}}{u_{r_1}}\right)( u_{r_2}- u_{r_1})\, dm_{r_1}\right|\\
\leq& 2\int_{r_1}^{r_2}\int |\Delta_ru_r|^2\, dm_{r_1}\, dr+(r_2-r_1)\int\left(\frac{g_{r_1}}{u_{r_1}}\Delta_{r_1}(u_{r_2}+u_{r_1})\right)^2\, dm_{r_1}\\
&+(r_2-r_1)\int  \Gamma_{r_1}(u_{r_2}+ u_{r_1})\Gamma_{r_1}\left(\frac{g_{r_1}}{u_{r_1}}\right)\, dm_{r_1}.
\end{align*}

This proves absolute continuity and together with \eqref{eq: ineq} we obtain
\begin{align}\label{der local log}
0\geq& \int_s^t\left(\int g_r\frac{\Gamma_r( u_r)}{u_r}\, dm_r-\int g_s\frac{\Gamma_s( u_s)}{u_s}\, dm_s\right)\, dr\\
\nonumber=&\int_s^t\int_s^r\frac{d}{dq}\int g_q\frac{\Gamma_q(u_q)}{u_q}\, dm_q\, dq\, dr.
\end{align}
The almost everywhere derivative is given by
\begin{align*}
\frac{d}{dq}\int g_q\frac{\Gamma_q(u_q)}{u_q}\, dm_q
=&
\lim_{h\to 0}\frac1h\left(\int g_{q+h}\frac{\Gamma_{q+h}(u_{q+h})}{u_{q+h}}\, dm_{q+h}
-\int g_q\frac{\Gamma_q(u_q)}{u_q}\, dm_q\right)\\
=&-\int \Delta_qg_q\frac{\Gamma_q(u_q)}{u_q}\, dm_q
-\int \Delta_qu_q \Gamma_q(u_q)\frac{g_q}{u_q^2}\, dm_q\\
&-2\int (\Delta_qu_q)^2\frac{g_q}{u_q}\, dm_q-2\int \Gamma_q\left(u_q,\frac{g_q}{u_q}\right)\Delta_qu_q\, dm_q\\
&+\int \frac{g_q}{u_q}(\partial_q\Gamma_q)(u_q)\, dm_q,
\end{align*}
where we used $u\geq c$, $u_q,g_q\in H^1((s,t);L^2(X))$, Proposition \ref{prop:prop}, \textbf{(A3)}, \textbf{(A2.a-c)}, and Lemma \ref{lma: con delta}.

Together with \eqref{der local log} we get 
\begin{align*}
0\geq&\int_s^t\int_s^r\Big[-\int \Delta_qg_q\frac{\Gamma_q(u_q)}{u_q}\, dm_q
-\int \Delta_qu_q \Gamma_q(u_q)\frac{g_q}{u_q^2}\, dm_q
-2\int (\Delta_qu_q)^2\frac{g_q}{u_q}\, dm_q\\
&-2\int \Gamma_q\left(u_q,\frac{g_q}{u_q}\right)\Delta_qu_q\, dm_q+\int \frac{g_q}{u_q}(\partial_q\Gamma_q)(u_q)\, dm_q\Big]\, dq\, dr\\
=&\int_s^t(t-q)\Big[-\int \Delta_qg_q\frac{\Gamma_q(u_q)}{u_q}\, dm_q
-\int \Delta_qu_q \Gamma_q(u_q)\frac{g_q}{u_q^2}\, dm_q
-2\int (\Delta_qu_q)^2\frac{g_q}{u_q}\, dm_q\\
&-2\int \Gamma_q\left(u_q,\frac{g_q}{u_q}\right)\Delta_qu_q\, dm_q+\int \frac{g_q}{u_q}(\partial_q\Gamma_q)(u_q)\, dm_q\Big]\, dq.
\end{align*}
Define
\begin{align*}
\Phi(q)=&-\int \Delta_qg_q\frac{\Gamma_q(u_q)}{u_q}\, dm_q
-\int \Delta_qu_q \Gamma_q(u_q)\frac{g_q}{u_q^2}\, dm_q
-2\int (\Delta_qu_q)^2\frac{g_q}{u_q}\, dm_q\\
&-2\int \Gamma_q\left(u_q,\frac{g_q}{u_q}\right)\Delta_qu_q\, dm_q+\int \frac{g_q}{u_q}(\partial_q\Gamma_q)(u_q)\, dm_q
\end{align*}

We want to show that $\Phi\colon[s, t]\to\mathbb R$ defines a continuous function. 
In order to do so, we consider each term separately.

 The first term $q\mapsto \int\Delta_q(g_q)\frac{\Gamma_q(u_q)}{u_q}\, dm_q$ is continuous since $q\mapsto \Delta_qg_q$, and $q\mapsto u_{q}^{-1}$ are continuous in $L^2(X)$ by \eqref{Laplace cont} and \eqref{heat-def2}, $q\mapsto \Gamma_q(u_{q})$ is weak$^*$ continuous in $L^\infty(X)$ by  Lemma \ref{lma: con delta} and \eqref{a priori grad}.
  
The second term  $q\mapsto \int\Delta_q(u_q)\frac{\Gamma_q(u_q)}{u^2_q}g_q\, dm_q$ is continuous since $q\mapsto \Delta_qu_q$ is continuous in $L^2(X)$ by \eqref{Laplace cont},
$q\mapsto \frac{g_q}{u_{q}^{2}}$ is continuous in $L^\infty(X)$ by \eqref{heat-def2} and \eqref{eq: cont bounded}, and $q\mapsto \Gamma_q(u_{q})$ is weak$^*$ continuous in $L^\infty(X)$ by  Lemma \ref{lma: con delta} and \eqref{a priori grad}.

The third term $q\mapsto \int (\Delta_q u_{q})^2\frac{g_q}{u_q}\, dm_q$ is continuous since $q\mapsto \Delta_q u_{q}$ is continuous in $L^2(X)$ by \eqref{Laplace cont}, and $q\mapsto \Gamma_q(u_{q})$ is weak$^*$ continuous by \eqref{a priori grad} and Lemma \ref{lma: con delta}, and $q\mapsto \frac{g_q}{u_q}$ are continuous in $L^\infty(X)$ by \eqref{eq: cont bounded}. 

The fourth term $q\mapsto \int \Gamma_q( u_{q},\frac{g_q}{u_q})\Delta_q u_{q}\, dm_q$ is continuous since $q\mapsto \Delta_q u_{q}$ is continuous in $L^2(X)$ and weak$^*$-continuous in $L^\infty(X)$ by \eqref{Laplace cont}, and $q\mapsto \Gamma_q(u_{q},\frac{g_q}{u_q})$ is continuous in $L^1(X)$ by \textbf{(A3.a)} and Lemma \ref{lma: con delta}. 

The last term $q\mapsto \int \frac{g_q}{u_q}(\partial_q\Gamma_q)(u_q)\, dm_q$ is continuous since $q\mapsto\frac{ g_q}{u_q}$ is continuous in $L^\infty(X)$ by \eqref{eq: cont bounded} and  $q\mapsto(\partial_q\Gamma_q)(u_q)$ is continuous in $L^1(X)$ by \textbf{(A2.c)} and $q\mapsto e^{-f_q}$ is continuous in $L^\infty(X)$.

Then it holds by Lebesgue differentiation
\begin{align*}
0\geq& \int \frac{g_s}{u}(\partial_s\Gamma_s)(u)\, dm_s-\int \Delta_sg_s\frac{\Gamma_s(u)}{u}\, dm_s
-\int \Delta_su \Gamma_s(u)\frac{g_s}{u^2}\, dm_s\\
&-2\int (\Delta_su)^2\frac{g_s}{u}\, dm_s
-2\int \Gamma_s\left(u,\frac{g_s}{u_s}\right)\Delta_su\, dm\\
=&\int ug_s(\partial_s\Gamma_s)(\log u)-\Delta_s(ug_s)\Gamma_{s}(\log u_s)\, dm_s-2\int (\Delta_s\log u)^2ug_s+\Gamma_s(\log u,ug_s)\Delta_s \log u\, dm_s,
\end{align*}
where we used the chain rule in the last equation.

Similarly as before we let $s\to t$ and obtain after choosing $\tilde g=e^{-v}g\in \D(\Delta_t)\cap L^\infty(X)$ and obtain recalling $u=e^v$
\begin{align*}
0\geq\int \tilde g\partial_t\Gamma_t(v)-\Delta_t(\tilde g)\Gamma_{t}(v)\, dm_t-2\int (\Delta_t v)^2\tilde g+\Gamma_t(v,\tilde g)\Delta_t v\, dm_t
\end{align*}
for all $v,\Delta_t v\in\D(\Delta_t)\cap\Lip_b(X)$ and $\tilde g\in \D(\Delta_t)\cap L^\infty(X)$ with $\tilde g\geq0$. The result for general $v\in\D(\Delta_t)\cap L^\infty(X)$ such that $\Gamma_t(v)\in L^\infty(X)$ and all $\tilde g\in\D(\Delta_t)\cap L^\infty(X)$ with $\tilde g\geq0$ follows by approximation with the semigroup mollifier from Definition \ref{molli}.

Similarly one deduces Bochner from the reverse local logarithmic Sobolev bound. Indeed by \eqref{eq:lower local log} it holds by the same argument as above
\begin{align*}
0\leq\int_s^t\int g_ru_r\Gamma_r(\log u_r)\, dm_r-\int g\frac{\Gamma_t(u_t)}{u_t}\, dm_t\, dr
\end{align*}
and since $q\mapsto\int g_qu_q\Gamma_q(\log u_q)\, dm_q$
\begin{align*}
0\geq\int_s^t\frac{d}{dq}\int_r^t\int g_qu_q\Gamma_q(\log u_q)\, dm_q\, dq\, dr,
\end{align*}
which is the same as in line \eqref{der local log}.
\end{proof}

\section{The dimension independent Harnack inequality}

\subsection{From $L^1$-gradient estimate to dimension independent  Harnack inequality}

This section will be devoted to derive the following result.
\begin{thm}\label{thm: thm}
Fix $\alpha>1$. Suppose that the $L^1$-gradient estimate \eqref{l1gradest}
holds.
Then for all $u\in L^2(X)$ such that $u\geq 0$, $t>s$ and $m$-a.e. $x,y\in X$ we have
 \begin{align}\label{eq: correct Harnack}
  (P_{t,s}u)^\alpha(y)\leq (P_{t,s}u^\alpha)(x)\exp\left\{\frac{\alpha d_t^2(x,y)}{4(\alpha-1)(t-s)}\right\}.
 \end{align}
\end{thm}

\medskip

Before starting with the proof of this results, let us 
 recall the notion of \emph{regular curves} as introduced in \cite{agsbe} and refined in \cite{ams}, as well as the notion of 
\emph{velocity densities} taken from \cite{ams}.
A curve $(\mu_r)_{r\in[0,1]}$ with $\mu_r=\rho_rm$ is called \emph{regular} if the following
are satisfied:
\begin{itemize}
 \item $\mu\in \Lip([0,1];(\mathcal P_2(X),W)) \cap \cC^1([0,1]; L^1(X))$ 
 \item There exists a constant $R>0$ such that $\rho_r\leq R$ $m$-a.e. for every $s\in[0,1]$
 \item $\sqrt{\rho_r}\in\D (\E)$ such that $\E(\sqrt \rho_r)\leq E$ for every $s\in [0,1]$.
\end{itemize}
We recall the following result (Lemma 12.2 in \cite{ams}).
\begin{lma}
 For every geodesic $(\mu_r)_{r\in[0,1]}$ there exist regular curves $\mu^n$ such that $\mu_r^n\to \mu_r$ in $L^2$-Kantorovich sense for all $r\in[0,1]$ and
 \begin{align*}
  \limsup_n\int_0^1|\dot\mu_r^n|^2\, dr\leq W^2(\mu_0,\mu_1).
 \end{align*}
\end{lma}

A regular curve $\mu$ admits a \emph{velocity density} $v\in L^2(X\times[0,1],\int\mu_t\, dt)$ in the sense that for every $\varphi\in\F$
\begin{align}\label{eq: regular}
 \left|\int\varphi\, d\mu_t-\int\varphi\, d\mu_s\right|\leq\int_s^t \int\sqrt{\Gamma(\varphi)}v_r\, d\mu_r\, dr
\end{align}
and there exists a unique velocity density with minimal $L^2(X\times[0,1],\int\mu_t\, dt)$-norm satisfying
\begin{align*}
 |\dot\mu_t|^2=\int v_t^2\, d\mu_t \quad\text{ for a.e. } t\in[0,1],
\end{align*}
see Theorem 6.6 and Lemma 8.1 in \cite{ams}.

\begin{proof}[Proof of Theorem \ref{thm: thm}]
Let $u\in L^2(X)\cap L^\infty(X)$, with $u\leq M$ $m$-a.e..
 Fix $s<t$ and define for all $s<r<t$
 \begin{align*}
 \psi_r^\varepsilon(u)&:=P_{t,r}\eta_\varepsilon(P_{r,s}u)\\
  \Psi^\varepsilon(r)&:=\int \omega_\varepsilon(\psi_r^\varepsilon(u))\, d\mu_r,
 \end{align*}
 where $\mu_r=\rho_rm_t$ is a regular curve in $\mathcal P_2(X)$, and $\omega_\eps, \eta_\eps$ are functions on $\mathbb R$ given
 \begin{align*}
  \eta_\varepsilon(z)=(z+\varepsilon)^\alpha-\varepsilon^\alpha,\qquad \omega_\eps(z)=\log(z+\varepsilon), \qquad 0<\eps<1.
 \end{align*}
Note that $\eta_\eps,\eta_\eps',\omega_\eps, \omega'_\eps\in\Lip_b([0,M])$ and
\begin{align}\label{etaeq}
 \eta_\eps(z)+\eps\geq (z+\eps)^\alpha, \qquad \eta_\eps(z)\leq z^\alpha.
\end{align}
Then we readily find that 
\begin{align}\label{eq: psi}
r\mapsto \psi_r^\eps(u)\in \cC([s,t]; L^2(X))
\end{align}
by \eqref{heat-def2}, \eqref{ad-heat-def2}, \eqref{eq: adjoint}, Lemma \ref{lma: maximum} and $\eta_\eps\in\Lip_b([0,M])$.

We claim that $r\mapsto \Psi^\varepsilon(r)$ is locally absolutely continuous. To see this we write
\begin{equation}
\begin{aligned}\label{eq:splitpsi}
|\Psi^\varepsilon(r+h)-\Psi^\varepsilon(r)|\leq &|\int \omega'_\varepsilon(\psi_\zeta^\varepsilon(u))(\psi_{r+h}^\varepsilon(u)-\psi_r^\varepsilon(u))\, d\mu_{r+h}|\\
&+\int_r^{r+h}\int|\omega'_\varepsilon(\psi_{r}^\eps(u))|\sqrt{\Gamma_t(\psi_{r}^\eps(u))}v_s\, d\mu_s\, ds,
\end{aligned}
\end{equation}
where $\zeta, \xi\in (r,r+h)$ and $v$ is the unique velocity density of $\mu$. The first term we estimate by
\begin{align*}
&\left|\int \omega'_\varepsilon(\psi_\zeta^\varepsilon(u))(\psi_{r+h}^\varepsilon(u)-\psi_r^\varepsilon(u))\, d\mu_{r+h}\right|\\
=&\left|\int P_{t,r+h}^*(\frac{\rho_{r+h}}{\psi^\eps_\zeta(u)+\eps})\eta_\eps(P_{r+h,s}u)\, dm_{r+h}-\int P^*_{t,r}(\frac{\rho_{r+h}}{\psi_\zeta^\eps(u)+\eps})\eta_\eps(P_{r,s}u)\, dm_r\right|\\
\leq &\left|\int P_{t,r+h}^*(\frac{\rho_{r+h}}{\psi^\eps_\zeta(u)+\eps})\eta_\eps(P_{r+h,s}u)\, dm_{r+h}-\int P^*_{t,r}(\frac{\rho_{r+h}}{\psi_\zeta^\eps(u)+\eps})\eta_\eps(P_{r+h,s}u)\, dm_r\right|\\
+&\left|\int P_{t,r}^*(\frac{\rho_{r+h}}{\psi^\eps_\zeta(u)+\eps})\eta_\eps(P_{r+h,s}u)\, dm_{r}-\int P^*_{t,r}(\frac{\rho_{r+h}}{\psi_\zeta^\eps(u)+\eps})\eta_\eps(P_{r,s}u)\, dm_r\right|\\
\leq &\int_r^{r+h}\int \left|\Gamma_q\left(P_{t,q}^*\left(\frac{\rho_{r+h}}{\psi^\eps_\zeta(u)+\eps}\right),\eta_\eps(P_{r+h,s}u)\right)\right|\, dm_{q}\, dq\\
&+\int_r^{r+h}\int\left|\Gamma_q\left( P_{t,r}^*\left(\frac{\rho_{r+h}}{\psi^\eps_\zeta(u)+\eps}\right)\eta'_\eps(P_{\xi,s}u) e^{f_q-f_r},P_{q,s}u\right)\right|\, dm_{q}\, dq,
\end{align*}
where we used \eqref{heat-def2} and \eqref{ad-heat-def2}, the 2-absolute continuity of $r\mapsto P^*_{t,r}g$, $r\mapsto P_{r,s}u$ by Proposition \ref{prop:prop}, the Lipschitz continuity of $\eta_\eps$, and the Lipschitz continuity of $r\mapsto f_r$. A calculation shows that this term is finite for almost all $s<r<h$.

For the second term in \eqref{eq:splitpsi} note that $|\omega_\eps'(\psi_\eps^r(u))|$ is uniformly bounded for almost all $s<r<t$, and by virtue of the $L^1$-gradient estimate \eqref{l1gradest}

\begin{equation*}
\begin{aligned}
&\sqrt{\Gamma_t(P_{t,r}\eta_\eps(P_{r,s}u))}\leq P_{t,r}\sqrt{\Gamma_r(\eta_\eps(P_{r,s}u))}
\leq P_{t,r}\left(\eta'_\eps(P_{r,s}u)\sqrt{\Gamma_r(P_{r,s}u)}\right),
\end{aligned}
\end{equation*}
which is an $L^2$-function on $(s,t)\times X$. All in all this proves the locally absolute continuity of $\Psi^\eps$.

In the next step we calculate the derivative of $\Psi^\eps(r)$. We compute
\begin{equation}
\begin{aligned}\label{eq: diff}
&\frac1h \int \omega_\varepsilon(\psi_{r+h}^\varepsilon(u))\, d\mu_{r+h}-\int \omega_\varepsilon(\psi_{r}^\varepsilon(u))\, d\mu_r\\
\leq &\frac1h \left(\int P_{t,r+h}^*\left(\frac{\rho_{r+h}}{\psi_\zeta^\eps(u)+\eps}\right)\eta_\eps( P_{r,s}u)\, dm_{r+h}-\int P_{t,r}^*\left(\frac{\rho_{r+h}}{\psi_\zeta^\eps(u)+\eps}\right)\eta_\eps( P_{r,s}u)\, dm_{r}\right)\\
&+ \frac1h \int P_{t,r+h}^*\left(\frac{\rho_{r+h}}{\psi_\zeta^\eps(u)+\eps}\right)(\eta_\eps( P_{r+h,s}u)-\eta_\eps(P_{r,s}u))\, dm_{r+h}\\
&+\frac1h\int_r^{r+h}\int \sqrt{\Gamma_t(\omega_\eps(\psi_q^\eps(u)))}v_q\, d\mu_q\, dq,
\end{aligned}
\end{equation}
where we used \eqref{eq: regular} for the last term.
Taking the limit $h\to0$,  by Proposition \ref{prop:prop} we get for the first term on the right hand side in \eqref{eq: diff}
\begin{align*}
&\lim_{h\to0}\frac1h \left(\int P_{t,r+h}^*\left(\frac{\rho_{r+h}}{\psi_\zeta^\eps(u)+\eps}\right)\eta_\eps( P_{r,s}u)\, dm_{r+h}-\int P_{t,r}^*\left(\frac{\rho_{r+h}}{\psi_\zeta^\eps(u)+\eps}\right)\eta_\eps( P_{r,s}u)\, dm_{r}\right)\\
=& \int \Gamma_r\left(P_{t,r}^*\left(\frac{\rho_{r}}{\psi_r^\eps(u)+\eps}\right),\eta_\eps( P_{r,s}u)\right)\, dm_r\\
 & + \lim_{h\to0}\frac1h \int\left(\frac{\rho_{r+h}}{\psi_\zeta^\eps(u)+\eps}-\frac{\rho_{r}}{\psi_r^\eps(u)+\eps}\right)\left(P_{t,r+h}(\eta_\eps( P_{r,s}u))-P_{t,r}(\eta_\eps( P_{r,s}u))\right)\, dm_t.
\end{align*}
Note that the last term is equal to $0$. Indeed, on the one hand $\frac1h(\frac{\rho_{r+h}}{\psi_\zeta^\eps(u)+\eps}-\frac{\rho_{r}}{\psi_r^\eps(u)+\eps})$ is a converging sequence in $L^1(X)$ due to $\rho\in \cC^1([0,1]; L^1(X))$,  $\omega_\eps'\in\Lip_b([0,M])$, and \eqref{eq: psi}. On the other $P_{t,r+h}(\eta_\eps( P_{r,s}u))-P_{t,r}(\eta_\eps( P_{r,s}u))\to 0$ weakly$^*$ in $L^\infty(X)$ due to \eqref{ad-heat-def2}, \eqref{eq: adjoint}, Lemma \ref{lma: maximum}, and the Banach-Alaoglu theorem.

For the second term on the right hand side in \eqref{eq: diff} it holds
\begin{align*}
&\lim_{h\to0} \frac1h \int P_{t,r+h}^*\left(\frac{\rho_{r+h}}{\psi_\zeta^\eps(u)+\eps}\right)(\eta_\eps( P_{r+h,s}u)-\eta_\eps(P_{r,s}u))\, dm_{r+h}\\
&\leq 
\int P^*_{t,r}\left(\frac{\rho_r}{\psi^\eps_r(u)+\eps}\right)\eta_\eps'(P_{r,s}u)\Delta_rP_{r,s}u\, dm_r\\
 & +\lim_{h\to0} \frac 1h \int P_{t,r+h}^*\left(\frac{\rho_{r+h}}{\psi_\zeta^\eps(u)+\eps}-\frac{\rho_{r}}{\psi_r^\eps(u)+\eps}\right)(\eta_\eps( P_{r+h,s}u)-\eta_\eps(P_{r,s}u))\, dm_{r+h}
\end{align*}
for a.e. $r$, since $\eta'_\eps(P_{r,s}u)$ in $\Lip_b([0,M])$, $\frac1h(P_{r+h,s}u-P_{r,s}u)\to \Delta_rP_{r,s}u$ in $L^2(X)$ for a.e. $r$ and $P_{t,r+h}^*(\frac{\rho_r}{\psi_r^\eps(u)+\eps})\to P^*_{t,r}(\frac{\rho_r}{\psi_r^\eps(u)+\eps})$ weakly$^*$ in $L^\infty(X)$ due to the uniform boundedness. The last term is equal to $0$ since $P_{t,r+h}^*\left(\frac{\rho_{r+h}}{\psi_\zeta^\eps(u)+\eps}-\frac{\rho_{r}}{\psi_r^\eps(u)+\eps}\right)\to 0$ weakly$^*$ in $L^\infty(X)$ by the Banach Alaoglu theorem and since $\frac{\rho_{r+h}}{\psi_\zeta^\eps(u)+\eps}-\frac{\rho_{r}}{\psi_r^\eps(u)+\eps}\to0$ in $L^1(X)$ similarly as above, and $P_{t,r}^*$ is a continuous operator on $L^1(X)$ (Lemma \ref{lma: maximum}).

 For the third term in \eqref{eq: diff} we apply Young's inequality and \eqref{l1gradest} and note that $|\omega_\eps'(\psi_q^\eps(u))|$ and $P_{t,q}(|\eta_\eps'(P_{q,s}u)|^2)$ are uniformly bounded on $(s,t)\times X$. Moreover by virtue of the local 
 Poincar\'e inequality (Theorem \ref{thmlocpoinc})
 \begin{align*}
\Gamma_t(\omega_\eps(P_{t,q}\eta_\eps(P_{q,s}u)))\leq |\omega'_\eps(P_{t,q}\eta_\eps(P_{q,s}u))|^2\frac{||\eta_\eps(P_{q,s}u)||_\infty^2}{2(t-q)}
 \end{align*}
 is a locally integrable function on $(s,t)\times X$. Then the Lebesgue differentiation theorem applies and thus
 \begin{align*}
\lim_{h\to0} \frac1h\int_r^{r+h}\int \sqrt{\Gamma_t(\omega_\eps(\psi_q^\eps(u)))}v_q\, d\mu_q\, dq=\int \sqrt{\Gamma_t(\omega_\eps(\psi_r^\eps(u)))}v_r\, d\mu_r
 \end{align*}
 for a.e. $s<r<t$.

 Summarizing we find by taking the limit in \eqref{eq: diff}
 \begin{align*}
 \frac{d}{dr}\Psi^\varepsilon(r)\leq &\int\Gamma_r\left(P_{t,r}^*\left(\frac{\rho_r}{\psi^\eps_r+\varepsilon}\right),\eta_\eps(P_{r,s}u)\right)+P_{t,r}^*\left(\frac{\rho_r}{\psi^\eps_r+\varepsilon}\right)\eta_\eps'(P_{r,s}u)\Delta_rP_{r,s}u\, dm_r\\
 &+\int \sqrt{\Gamma_t(\omega_\eps(\psi_r^\eps(u)))}v_r\, d\mu_r\\
 =&-\int\eta_\eps''(P_{r,s}u)\Gamma_r(P_{r,s}u)P^*_{t,r}\left(\frac{\rho_r}{\psi_r^\eps+\varepsilon}\right)\, dm_r
 +\int|\omega'_\eps(\psi_r^\eps(u)))| \sqrt{\Gamma_t(\psi_r^\eps(u))}v_r\, d\mu_r,
\end{align*}
where we used integration by parts and the chain rule in the last line.

Applying the gradient estimate \eqref{l1gradest}, using the chain rule twice, and inserting the definitions we compute
\begin{align*}
&\frac{d}{dr}\Psi^\varepsilon(r) \\
\leq &-\int\eta_\eps''(P_{r,s}u)\Gamma_r(P_{r,s}u)P_{t,r}^*\left(\frac{\rho_r}{\psi_r^\eps(u)+\varepsilon}\right)\, dm_r
 +\int \frac{\rho_r}{\psi_r^\eps(u)+\varepsilon} P_{t,r}\left(\eta_\eps'(P_{r,s}u)\sqrt{\Gamma_r(P_{r,s}u)}\right)v_r\, dm_t\\
 =&\int\frac{\alpha\rho_r}{\psi^\eps_r(u)+\varepsilon}\left(-(\alpha-1)P_{t,r}\left((P_{r,s}u+\eps)^\alpha\frac{\Gamma_r(P_{r,s}u)}{(P_{r,s}u+\eps)^{2}}\right)+v_rP_{t,r}\left((P_{r,s}u+\eps)^{\alpha}\frac{\sqrt{\Gamma_r(P_{r,s}u)}}{P_{r,s}u+\eps}\right)\right)\, dm_t\\
 \leq&\int\frac{\alpha\rho_r}{\psi^\eps_r(u)+\varepsilon}\sup_\kappa\{-(\alpha-1)P_{t,r}(P_{r,s}u+\eps)^\alpha\kappa^2+v_rP_{t,r}(P_{r,s}u+\eps)^\alpha\kappa\}\, dm_t
\end{align*}
Calculating the supremum and using \eqref{etaeq} further yields
\begin{align*}
&\frac{d}{dr}\Psi^\varepsilon(r) 
 \leq\int\frac{\alpha\rho_rP_{t,r}(P_{r,s}u+\eps)^\alpha}{\psi^\eps_r(u)+\varepsilon}\frac{v_r^2}{4(\alpha-1)}\, dm_t
\leq\frac{\alpha}{4(\alpha-1)}\int v_r^2\, d\mu_r=\frac{\alpha}{4(\alpha-1)}|\dot\mu_r|^2,
\end{align*}
where we used that $v$ is the minimal velocity density for $\mu$.

Due to the local absolute continuity, integrating from $s$ to $t$ yields
\begin{align*}
 \Psi_\varepsilon(t)-\Psi_\varepsilon(s)\leq \frac{\alpha}{4(\alpha-1)}\int_s^t|\dot\mu_r|^2\, dr.
\end{align*}
Hence, by approximating $W_t^2$-geodesics with regular curves and taking the scaling into account we end up with
\begin{align*}
 \Psi_\varepsilon(t)-\Psi_\varepsilon(s)\leq \frac{\alpha}{4(\alpha-1)(t-s)}W_t(\mu_s,\mu_t)^2.
\end{align*}
We get for $m$-a.e. $x,y\in X$, after letting $\mu_s\to\delta_x$ and $\mu_t\to \delta_y$ with respect to $L^2$-Kantorovich distance, 
\begin{align*}
 &\log\frac{\eta_\eps(P_{t,s}u)(y)}{P_{t,s}\eta_\eps(u)(x)}\leq\frac{\alpha d_t^2(x,y)}{4(\alpha-1)(t-s)}.
\end{align*}
Now we let $\eps\to0$. Since $\eta_\eps(P_{t,s}u)\to (P_{t,s}u)^\alpha$, and $P_{t,s}\eta_\eps(u)\to P_{t,s}(u^\alpha)$ a.e. by monotone convergence we find
\begin{align*}
 &\frac{(P_{t,s}u)^\alpha(y)}{P_{t,s}(u^\alpha)(x)}\leq\exp\left\{\frac{\alpha d_t^2(x,y)}{4(\alpha-1)(t-s)}\right\},
\end{align*}
which is the result for $u\in L^2(X)\cap L^\infty(X)$. The result for general $u$ follows by a truncation argument.
\end{proof}

\subsection{From dimension independent Harnack inequality  to local logarithmic Sobolev inequality}
We assume in this section that $m_t(X)<\infty$ for some and thus for all $t\in(0,T)$.
\begin{thm}
Assume that the Harnack inequality \eqref{eq: correct Harnack} holds. Then for all $u\in\D(S)\cap L^1(X)$ such that $u\geq 0$ the local logarithmic Sobolev inequality holds
\begin{align*}
P_{t,s}(u\log u)-P_{t,s}u\log P_{t,s}u\geq (t-s)\frac{\Gamma_t(P_{t,s}u)}{P_{t,s}u},
\qquad m\text{-a.e..}
\end{align*}

\end{thm}

\begin{proof}
Let $u\in L^1(X)\cap L^\infty(X)$ with $u\geq c>0$.
From the Harnack inequality it follows that
\begin{align}\label{integratedharnack}
\int\alpha\log(P_{t,s}u)\, d\mu- \int\log(P_{t,s}(u^\alpha))\, d\nu\leq\frac{\alpha W_t^2(\mu,\nu)}{4(\alpha-1)(t-s)}
\end{align}
holds
for each probability measures $\mu,\nu$ which are absolutely continuous with respect to $m_t$. This follows from integrating \eqref{eq: correct Harnack} with respect to an optimal transport plan.

Now choose $\mu=g m_t$ with $g\geq0$ and $g\in\F\cap L^\infty(X)$. Consider the associated Dirichlet form $\E^g(u):=\int \Gamma_t(u)g \, dm_t$ with heat semigroup $(H_r^g)_{r\geq0}$ and generator $\Delta^g$. We introduce for fixed $\varepsilon>0$ the function 
\begin{align*}
 \psi=\frac1\varepsilon\int_0^\infty H_r^g(\psi_0)\kappa(r/\varepsilon)\, dr,
\end{align*}
 where $\kappa\in\cC_c^\infty(0,\infty)$ with $\kappa\geq0$ and $\int_0^\infty\kappa(r)\, dr=1$ and $\psi_0\in \D(\E^g)\cap L^\infty(gm_t)$. Note that $||\Delta^g\psi||_\infty\leq M$ for some $M\geq0$ and hence $\mu_\tau:=g(1-\tau\Delta^g\psi)m_t$ is a probability measure for all $\tau<1/2M$. First we will show that 
\begin{align}\label{W,G}
\limsup_{\tau\to0}\frac{1}{2 \tau^2}W^2_t(\mu,\mu_\tau)\leq \frac12\int \Gamma_t(\psi)g\, dm_t
 \end{align}
 using the Hopf-Lax semigroup $(Q_r)_{r\geq0}$ with respect to $d_t$. For $\varphi\in\mathcal C_b(X)$ we find for $r\leq \tau$
 \begin{align*}
 \frac{d}{dr}\int Q_r(\varphi)\, d\mu_r\leq &\int(-\frac12\Gamma_t(Q_r(\varphi))(1-\tau\Delta^g\psi)-Q_r(\varphi)\Delta^g\psi)g\, dm_t\\
 \leq& \int(-\frac12\Gamma_t(Q_r(\varphi))(1-\tau M)+\Gamma_t(Q_r(\varphi),\psi)g\, dm_t\\
 \leq &\frac{1}{2(1-\tau M)}\int \Gamma_t(\psi)g\, dm_t.
 \end{align*}
 Integrating on $[0,\tau]$, taking the supremum over all $\varphi$, dividing by $\tau$ and letting $\tau\to0$ yields \eqref{W,G}.
   For $\alpha=1+\tau$, $\tau>0$ \eqref{integratedharnack} reads as
\begin{align}\label{tauformula}
(1+\tau)\int\log(P_{t,s}u)\, d\mu- \int\log(P_{t,s}(u^{1+\tau}))\, d\mu_\tau\leq\left\{\frac{(1+\tau) W_t^2(\mu,\mu_\tau)}{4\tau(t-s)}\right\}.
 \end{align}
 We divide
 by $\tau>0$ and let $\tau\to0$. By \eqref{W,G} the right hand side can be estimated from above by
 \begin{align*}
 \frac1{4(t-s)}\int \Gamma_t(\psi)g\, dm_t.
 \end{align*}
 We claim
 that together with the left hand side this amounts to 
 \begin{align}\label{formula}
 \int \log(P_{t,s}u)\, d\mu-\int \frac{P_{t,s}(u\log u)}{P_{t,s}u}\, d\mu-\int\Gamma_t(\log(P_{t,s}u),\psi)\, d\mu\leq \frac1{4(t-s)}\int\Gamma_t(\psi)\, d\mu.
 \end{align}
 Indeed, it is straight forward to check that $r\mapsto\int \log P_{t,s} u^{1+r}\, d\mu_r$ is absolutely continuous with derivative 
 \begin{align*}
 \Psi(r):=\int\frac{P_{t,s}(u^{1+r}\log u)}{P_{t,s}u^{1+r}}\, d\mu_r-\int \log P_{t,s} u^{1+r}(\Delta_t^g\psi)g\, dm_t.
 \end{align*}
Since $u\geq c>0$ we see that $r\mapsto\Psi(r)$ is continuous. Hence
\begin{align*}
 &\frac1\tau(\int\log(P_{t,s}u)\, d\mu- \int\log(P_{t,s}(u^{1+\tau}))\, d\mu_\tau)
 =-\frac1\tau\int_0^\tau \Psi(r)\, dr\\
 &\xrightarrow{\tau\to0} -\int \frac{P_{t,s}(u\log u)}{P_{t,s}u}\, d\mu-\int\Gamma_t(\log(P_{t,s}u),\psi)\, d\mu.
\end{align*}
Together with \eqref{tauformula} this yields \eqref{formula}.

 Letting $\varepsilon\to0$ we conclude 
  \begin{align*}
 \int \log(P_{t,s}u)\, d\mu-\int \frac{P_{t,s}(u\log u)}{P_{t,s}u}\, d\mu-\int\Gamma_t(\log(P_{t,s}u),\psi_0)\, d\mu\leq \frac1{4(t-s)}\int\Gamma_t(\psi_0)\, d\mu.
 \end{align*}

Now we may choose 
$\psi_0=-2(t-s)\log(P_{t,s}u)$ and obtain
  \begin{align*}
 \int \log(P_{t,s}u)\, d\mu-\int \frac{P_{t,s}(u\log u)}{P_{t,s}u}\, d\mu+(t-s)\int\Gamma_t(\log(P_{t,s}u))\, d\mu\leq 0.
 \end{align*}
Since this holds for all $\mu=g m_t$, we recover the local logarithmic Sobolev inequality
\begin{align*}
P_{t,s}(u\log u)-P_{t,s}u\log P_{t,s}u\geq (t-s)\frac{\Gamma_t(P_{t,s}u)}{P_{t,s}u},
\end{align*}
for all $u\in L^1(X)\cap L^\infty(X)$ with $u\geq c>0$. We obtain the estimate for all nonnegative
$u\in\D(S)\cap L^1(X)$ by a truncation argument.
\end{proof}


\section{The logarithmic Harnack inequality}

We already noted in Remark 1.5, that the dimension-independent Harnack inequality (for some exponent $\alpha$) implies the logarithmic Harnack inequality. 

This section is devoted to prove that
the logarithmic Harnack inequality  implies the dynamic Bochner inequality. To do so, in addition to our standing assumptions, in particular, the validity of a RCD$(K,\infty)$-condition for each $(X,d_t,m_t)$ and a log-Lipschitz dependence on $t$ for $d_t$ and $m_t$, we have to impose various  continuity assumptions (all of which are satisfied in the static case).

We assume that $m_t(X)<\infty$ for $t\in(0,T)$, (\textbf{A2.a-c}), and (\textbf{A3})
hold.
Moreover, writing $u_{q,s}=P_{q,s}u$, we assume that
%
\begin{itemize}
 \item for $u\in \F\cap\D(\Delta)$ the functions
 \begin{align}
&q\mapsto u_{q,s}, \quad s\mapsto \Delta_su, \quad
q\mapsto \Delta_q u_{q,s}
\label{ass1}\tag{{\bf A5.a}}
\end{align}
are continuous in $\F\cap L^1(X)$;
\item for $w,w_q\in\D(\Delta)$ 
as $q\to t$, and $\Delta_tw_q\to \Delta_tw$ in $L^1(X)$
\begin{align}
 &\Delta_qP_{t,q}^*w_q\to \Delta_t w\quad\mbox{in }L^1(X)
 \label{ass4}\tag{{\bf A5.b}}.
\end{align}
\end{itemize}

Let us emphasize that (\textbf{A5.a+b}) are always satisfied in the static case.


\begin{thm}\label{thm: logHartosrf}
If for all nonnegative $u\in L^1(X)\cap L^\infty(X)$ and $s<t$ the logarithmic Harnack inequality
\begin{align}\label{eq: correct log Harnack}
P_{t,s}(\log u)(x)\leq \log(P_{t,s} u)(y)+\frac{d_t^2(x,y)}{4(t-s)}
\end{align}
holds for $m$-a.e. $x,y\in X$, then the pointwise dynamic Bochner inequality holds at time $t$, i.e. 
\begin{align*}
 \frac12\int\Gamma_t(f)\Delta_t g\, dm_t+\int(\Delta_t f)^2g+\Gamma_t(f,g)\Delta_t f\, dm_t\geq\frac12 \int (\partial_t \Gamma_t)(f) g\, dm_t
\end{align*}
for all $f\in\D(\Delta_t)\cap L^\infty(X)$ such that $\Gamma_t(f)\in L^\infty(X)$ and all nonnegative $g\in\D(\Delta_t)\cap L^\infty(X)$.
\end{thm}

\begin{proof}

Let us introduce some function $g$ satisfying $C\geq g\geq c>0$.  Moreover we will assume that $g\in\D(\Delta_t)\cap \Lip(X)$ such that $\Delta_t g\in \F$. We define the Cheeger energy $\frac12\E^g_t$ associated with $d_t$ and finite measure $g m_t$. The operator $\Gamma_t( f)$ is invariant under this perturbations, hence $\Gamma_t^g (f)=\Gamma_t( f)$ and $\D(\E^g_t)=\F$. We refer to \cite[Section 4]{agscalc} for these facts. This leads to the following integral representation of $\E^g_t$
\begin{align*}
\E^g_t(f)=\int \Gamma_t(f)g\, dm_t,
\end{align*}
which makes it a symmetric bilinear form. We denote the associated (Markovian) semigroup by $P_s^g$ and its generator by $\Delta_t^g$, which satisfies the following integration by parts formula
\begin{align*}
\int \Delta_t^g f h g\, dm_t=-\int \Gamma_t(f,h)g\, dm_t
\end{align*}
for all $f\in \D(\Delta_t^g)$ and $h\in \D(\E_t^g)$. Since $\log g\in \F$ this can be rewritten into 
\begin{align*}
\Delta_t^g=\Delta_t+\Gamma_t(\log g,\cdot)
\end{align*}
and thus $\D(\Delta_t)\subset \D(\Delta_t^g)$.

Let $f,\Delta_t f\in D(\Delta_t)\cap \Lip_b(X)$. Then by Lemma \ref{regofexp} $u:=e^f\in \D(\Delta_t)\cap \Lip_b(X)$ with $\Delta_t e^f,\Delta^g_t e^f\in L^\infty(X)\cap\F$ and $u\geq e^{-||f||_\infty}=:\eps>0$.

For $s\le t$ we set 
$$v_s= P^g_{t-s}e^{-2f}\quad\mbox{and}\quad \mu_s= v_s g m_t.$$
Note that $v_s\in \D(\Delta_t^g)\cap L^\infty(X)$ for all $s\leq t$ by Lemma \ref{regofexp}.
Without restriction, we may assume that $\mu_t$, and hence $\mu_s$ for every $s<t$, is a probability measure. Otherwise, simply replace $f$ by $f+C$ for a suitable constant $C$.

Assume that  the logarithmic Harnack inequality holds for the function $u=e^f$. We integrate the inequality w.r.t.\ the $W_t$-optimal coupling of $\mu_t$ and $\mu_s$ to obtain for any $s<t$
%
\begin{align}\label{infdimHarnackbc}
\int P_{t,s}\log u\, d\mu_s-\int\log P_{t,s}u\, d\mu_t\leq\frac1{4(t-s)}W_t^2(\mu_t,\mu_s).
\end{align}

\medskip

Consider the map $r\mapsto \int P_{t,r}\log P_{r,s}u\, d\mu_r$. This map is absolutely continuous since for a.e. $s<r_1<r_2<t$
\begin{align*}
\left|\int P_{t,r_2}\log P_{r_2,s}u\, d\mu_{r_2}-\int P_{t,r_1}\log P_{r_1,s}u\, d\mu_{r_1}\right|
\leq&\left|\int_{r_1}^{r_2}\int \Gamma_r(\log u_{r_2},P_{t,r}^*(v_{r_2}g))\, dm_{r}\, dr\right|\\
&+\frac12\int_{r_1}^{r_2}\int|\Delta_r u_r|^2\, dm_{r_1}\, dr+\frac{(r_2-r_1)}{2\eps^2}||P_{t,r_1}^*(v_{r_2}g)||_2^2\\
&+\left|\int_{r_1}^{r_2}\int P_{t,r_1} \log u_{r_2}(\Delta_t^g v_{r})g\, dm_{t}\, dr\right|.
\end{align*}
Hence for the left hand side of \eqref{infdimHarnackbc} we find by differentiation
\begin{align*}
&\int P_{t,s}\log u\, d\mu_s-\int\log P_{t,s}u\, d\mu_t=-\int_s^t\frac{d}{dr}\int P_{t,r}\log P_{r,s}u\, d\mu_r\, dr\\
=&\int_s^t\int P_{t,r}\Delta_r\log P_{r,s}u-P_{t,r}\frac{\Delta_r P_{r,s}u}{P_{r,s}u}-\Gamma_t(P_{t,r}\log P_{r,s}u,\log v_{r})\, d\mu_r\, dr\\
=&-\int_s^t\int P_{t,r}\Gamma_r(\log P_{r,s}u)+\Gamma_t(P_{t,r}\log P_{r,s}u,\log v_{r})\, d\mu_r\, dr
\end{align*}
and for the right hand side Kuwada's Lemma (\cite[Lemma 6.1]{agscalc}) yields 
\begin{align*}
\frac1{4(t-s)}W_t^2(\mu_s,\mu_t)\leq
\frac14\int_s^t\int\Gamma_t(\log  v_{r})\, d\mu_r\, dr. 
\end{align*}
Hence \eqref{infdimHarnackbc} can be rewritten as follows
\begin{align}\label{infdimHarnackbc2}
\int_s^t\int-P_{t,r}\Gamma_r(\log P_{r,s}u)-\Gamma_t(P_{t,r}\log P_{r,s}u,\log v_{r})-\frac14\Gamma_t(\log v_{r})\, d\mu_r\, dr\leq 0.
\end{align}

\medskip

Now let us consider the map
\begin{align*}
r\mapsto&\int-P_{t,r}\Gamma_r(\log P_{r,s}u)-\Gamma_t(P_{t,r}\log P_{r,s}u,\log v_{r})-\frac14\Gamma_t(\log v_{r})\, d\mu_r\\
=:& I(r)+II(r)+III(r).
\end{align*}
From Lemma \ref{lma: differentiation} we know that the map $r\mapsto III(r)$ is absolutely continuous with derivative
\begin{align*}
\frac{d}{dr}III(r)=&\int \Big(\frac12\Gamma_t(\log v_{r},\Delta_t^g v_{r})-\frac14\Gamma_t(\log v_{r}){\Delta_t^g v_{r}}\Big){g}\, dm_t\\
=&\frac12\int\Gamma_t\left(\log v_{r},\frac{\Delta_t^g v_{r}}{v_{r}}\right)
-\frac14\Gamma_t\bigg(\Gamma_t(\log v_{r}),\log v_{r}\bigg)\, d\mu_{r}.
\end{align*}
For $I$ we calculate for a.e. $r_1<r_2$
\begin{align*}
|I(r_1)-I(r_2)|\leq& \left|\int_{r_1}^{r_2}\int \frac{\Gamma_{r_2}(u_{r_2,s})}{u_{r_2,s}^2}\Delta_rP_{t,r}^*(v_{r_2}g)\, dm_rdr\right|\\
&+\left|\int\left(\frac{\Gamma_{r_2}(u_{r_2,s})}{u_{r_2,s}^2}-\frac{\Gamma_{r_1}(u_{r_1,s})}{u_{r_1,s}^2}\right)P_{t,r_1}^*(v_{r_2}g)\, dm_{r_1}\right|\\
&+\left|\int_{r_1}^{r_2}\int P_{t,r_1}\left(\frac{\Gamma_{r_1}(u_{r_1,s})}{u_{r_1,s}^2}\right)\Delta_t^g v_{r}g\, dm_{r_1}\, dr\right|.
\end{align*}
The second term of this subdivision can be estimated as follows
\begin{align*}
&\left|\int\left(\frac{\Gamma_{r_2}(u_{r_2,s})}{u_{r_2,s}^2}-\frac{\Gamma_{r_1}(u_{r_1})}{u_{r_1}^2}\right)P_{t,r_1}^*(v_{r_2}g)\, dm_{r_1}\right|\\
\leq &\frac{C(r_2-r_1)}{\eps^2}\left|\int \Gamma_{r_1}(u_{r_2,s})P_{t,r_1}^*(v_{r_2}g)\, dm_{r_1}\right|\\
&+\frac1{\eps^2}\left|\int_{r_1}^{r_2}\int \Delta_ru_{r,s}\Delta_{r_1}(u_{r_2,s}+u_{r_1,s})P_{t,r_1}^*(v_{r_2}g)\, dm_{r_1}\, dr\right|\\
&+\frac1{\eps^2}\left|\int_{r_1}^{r_2}\int \Delta_ru_{r,s}\Gamma_{r_1}(P_{t,r_1}^*(v_{r_2}g),u_{r_2}+u_{r_1,s})\, dm_{r_1}\, dr\right|\\
&+\frac1{\eps^3}\left|\int_{r_1}^{r_2}\int \Gamma_{r_1}(u_{r_1,s})\Delta_ru_{r,s} P_{t,r_1}^*(v_{r_2}g)\, dm_{r_1}\, dr\right|.
\end{align*}
For the almost everywhere derivative we obtain by eventually using Proposition \ref{prop:prop}, Lemma \ref{lma: maximum}, Lemma \ref{lma: con delta}, \eqref{a priori grad}, \eqref{Laplace cont}, \eqref{eq: cont derivative}, and \eqref{eq: cont bounded}.
\begin{align*}
\frac{d}{dr}I(r)=&-\int \frac{\Gamma_r(u_{r,s})}{u_{r,s}}\Delta_r P^*_{t,r}(v_rg)\, dm_r
+\int (\partial_r\Gamma_r)(\log u_{r,s})P_{t,r}^*(v_rg)\, dm_r
-\int P_{t,r}\left(\frac{\Gamma_r(u_{r,s})}{u_{r,s}^2}\right)\Delta_t^gv_rg\, dm_t\\
&+2\int\Gamma_r(\Delta_r u_{r,s},u_{r,s})\frac{P^*_{t,r}(v_rg)}{u_{r,s}^2}\, dm_r-2\int\Gamma_r(u_{r,s})\frac{P^*_{t,r}(v_rg)}{u_{r,s}^3}\Delta_r u_{r,s}\, dm_r.
\end{align*}
Finally for $II$ we argue similarly as for $I$ and prove local absolute continuity by
\begin{align*}
|II(r_1)-II(r_2)|\leq \left|\int_{r_1}^{r_2}\int \log u_{{r,s}_2}\Delta_rP_{t,r}^*(\Delta_t^g v_{r_2}g)\, dm_r\, dr\right|\\
+\frac1\eps\left|\int_{r_1}^{r_2}\int \Delta_ru_{r,s}P_{t,r_1}^*(\Delta_t^g v_{r_2}g)\, dm_{r_1}\, dr\right|\\
+\left|\int_{r_1}^{r_2}\int\Gamma_t(P_{t,r_1}\log u_{r_1,s},\Delta_t^g v_{r})g\, dm_t\, dr\right|.
\end{align*}
For the almost everywhere derivative we obtain by eventually using Proposition \ref{prop:prop} and Lemma \ref{lma: maximum}
\begin{align*}
\frac{d}{dr}II(r)=\int \Gamma_t(P_{t,r}\log u_{r,s},\Delta_t^gv_r) g\, dm_t +\int\Delta_t P_{t,r}\log u_{r,s}\, \Delta_t^g v_r\,  g\, dm_t-\int P_{t,r}\frac{\Delta_r u_{r,s}}{u_{r,s}}\, \Delta_t^gv_r\, g \, dm_t.
\end{align*}

Thus $r\mapsto I(r)+II(r)+III(r)$ is absolutely continuous and we rewrite \eqref{infdimHarnackbc2} as 
\begin{equation}
\begin{aligned}\label{infdimHarnackb2c}
&\int_s^t\int_r^t\frac{d}{dq}\int P_{t,q}\Gamma_q(\log u_{q,s})+\Gamma_t(P_{t,q}\log u_{q,s},\log v_{q})+\frac14\Gamma_t(\log v_{q})\, d\mu_{q}\, dq\, dr\\
\leq& \int_s^t\int \Gamma_t(\log u_{t,s})+\Gamma_t(\log u_{t,s},\log  v_t)+\frac14\Gamma_t(\log  v_t)\, d\mu_t\, dr\\
=& (t-s)\int \Gamma_t(\log u_{t,s})+\Gamma_t(\log u_{t,s},\log  v_t)+\frac14\Gamma_t(\log  v_t)\, d\mu_t,
\end{aligned}
\end{equation}
where the right hand side comes from the boundary term $I(t)+II(t)+III(t)$.

Recall that $\mu_q= v_q\,g\, m_t$.
Then the term on the LHS of \eqref{infdimHarnackb2c} takes the form
\begin{align}
&\int_s^t\int_r^t\frac{d}{dq}\int\Big[P_{t,q}\Gamma_q(\log u_{q,s})+\Gamma_t(P_{t,q} \log u_{q,s},\log v_{q})+\frac14\Gamma_t(\log v_{q})\Big]\,  v_q\,g\, d m_t\, dq\, dr\nonumber\\
&=\int_s^t\int_r^t\Big[\int
-\Gamma_q(\log u_{q,s})\Delta_q P_{t,q}^*\big(v_q\,\,g\big)+\Big((\partial_q\Gamma_q)(\log u_{q,s})+2\Gamma_q\big(\log u_{q,s},\frac1{u_{q,s}}\Delta_q u_{q,s}\big)\Big)\,
P_{t,q}^*\big(v_q\,\,g\big)\,d m_q\nonumber\\
&\quad+\int
\Big(\Gamma_t(P_{t,q}\log u_{q,s},\Delta^g_t v_{q})
+\frac12\Gamma_t\Big(\log v_{q},\frac{\Delta^g_t v_{q}}{v_{q}}\Big)\, v_q+\frac14\Gamma_t(\log v_{q})\,
\Delta_t^g v_q\Big)\,\,g\,d m_t
\Big] dq\, dr\nonumber\\
&=:\int_s^t\int_r^t\Psi(q)\, dq\, dr=\int_s^t(t-q)\Psi(q)\, dq.
\label{Psi}
\end{align}

We decompose $\Psi$ into five terms and verify the continuity of each of them. For the first one,
\begin{align*}
\Psi_1(q):=-\int\Gamma_q(\log u_{q,s})\,\Delta_qP_{t,q}^*\big(v_q\,\,g\big)\,d m_q.
\end{align*}
continuity follows from
the fact that $q\mapsto\Gamma_q(\log u_{q,s})$ is weak$^*$-continuous in $L^\infty(X)$ by \eqref{ass1} and \eqref{a priori grad}, and $q\mapsto \Delta_q P_{t,q}^*(v_q g)$ is continuous in $L^1(X)$ by assumption \eqref{ass4} together
with the fact that $q\mapsto \Delta_t(v_q\,g)$ is continuous in $L^1(X)$.

%
%


Continuity of the second one,
\begin{align*}
\Psi_2(q):=\int(\partial_q\Gamma_q)(\log u_{q,s})\,P_{t,q}^*\big(v_q\,\,g\big)\,d m_q,
\end{align*}
follows from  $L^1$-continuity of $q\mapsto \partial_q\Gamma_q(\log u_{q,s})$, as requested in assumption \textbf{(A2.c)}, \textbf{(A3)}, and the weak$^*$-continuity  of $q\mapsto P_{t,q}^*\big(v_q\,\,g\big)$ in $L^\infty(X)$, resulting from \eqref{ass4} together with the uniform boundedness in $L^\infty(X)$.

For the third one,
\begin{align*}
\Psi_3(q)&:=2\int \Gamma_q\Big(\log u_{q,s},\frac1{u_{q,s}}\Delta_q u_{q,s}\Big)\,
P_{t,q}^*\big(v_q\,\,g\big)\,d m_q
\end{align*}
assumptions \textbf{(A2.b)}, \textbf{(A3)} and \eqref{ass1} yield continuity of $q\mapsto\Gamma_q\Big(\log u_{q,s},\frac1{u_{q,s}}\Delta_q u_{q,s}\Big)$ in $L^1(X)$ combined with \eqref{a priori grad} and dominated convergence. Together with the weak$^*$-continuity of 
$q\mapsto P_{t,q}^*\big(v_q\,\,g\big)$ in $L^\infty(X)$, this yields the claim.

The fourth term,
\begin{align*}
\Psi_4(q)&:=\int
\Gamma_t(P_{t,q}\log u_{q,s},\Delta^g_t v_{q})
\,\,g\,d m_t
\end{align*}
is continuous since $q\mapsto P_{t,q}\log u_{q,s}$ is continuous in $\F$ by \eqref{ass1} and \eqref{a priori grad}, and $q\mapsto \Delta^g_tv_q$ is continuous in $\F$ by Lemma \ref{regofexp}.

The final term 
\begin{align*}
\Psi_5(q)&:=\int\Big[
\frac12\Gamma_t(\log v_{q},\frac{\Delta^g_t v_{q}}{v_{q}})\, v_q+\frac14\Gamma_t(\log v_{q})\,
\Delta^g_t v_q\Big]\,\,g\,d m_t\\
&=\int\Big[-\frac1{2 v_q}\big(\Delta^g_t v_{q}\big)^2+\frac14\Gamma_t(\log v_q)\,\Delta^g_t v_{q}\Big]\,\,g\,d m_t
\end{align*}
is always continuous in $q$ without any extra assumption.

Similarly one computes the right hand side of \eqref{infdimHarnackb2c}. Recalling that $\log v_t=-2f$:
\begin{align}
&\frac1{t-s}\int\Big[\Gamma_t(\log u_{t,s})+\Gamma_t(\log u_{t,s} ,\log v_t)+\frac14\Gamma_t(\log  v_t)\Big]\, d\mu_t
=\frac1{t-s}\int\Gamma_t\Big(\log u_{t,s} -f\Big)\, d\mu_t
\nonumber\\
&=\frac{1}{t-s}\int_s^t\partial_q\int\Gamma_t\Big(\log u_{q,s} -f\Big)
\, d\mu_t\, dq\nonumber\\
&=\frac2{t-s}\int_s^t\int 
\Gamma_t\Big(\log u_{q,s} -f,
\frac{\Delta_q u_{q,s}}{u_{q,s}}
\Big)
\, d\mu_t\, dq\nonumber.
\end{align}
Note that by the continuity of $q\mapsto \log u_{q}$ in $\F$ and the continuity of
$q\mapsto \frac{\Delta_q u_{q}}{u_{q,s}}$ in $\F$ by virtue of \eqref{ass1}, \eqref{a priori grad} and the fact that $u\geq\eps$, the map $q\mapsto\int 
\Gamma_t\Big(\log u_{q,s} -f,
\frac{\Delta_q u_{q,s}}{u_{q,s}}
\Big)\, d\mu_t$ is continuous. Then by the Lebesgue differentiation theorem and the continuity discussion above we deduce from \eqref{infdimHarnackb2c} that 
 (recalling that $u=e^f$)
\begin{align*}
\Psi(s)=&\int
-\Gamma_s(f)\Delta_sP_{t,s}^*\big(v_s\,\,g\big)+\Big((\partial_s\Gamma_s)(f)+2\Gamma_s\big(f,\frac1{e^f}\Delta_s e^f\big)\Big)\,
P_{t,s}^*\big(v_s\,\,g\big)\,d m_s\nonumber\\
&\quad+\int\Big(
\Gamma_t(P_{t,s} f,\Delta^g_t v_{s})
+\frac12\Gamma_t(\log v_{s},\frac{\Delta^g_t v_{s}}{v_{s}})\, v_s+\frac14\Gamma_t(\log v_{s})\,
\Delta_t^g v_s\Big)\,\,g\,d m_t\leq 0.
\end{align*}

Then, letting $s\to t$, by continuity we have (recalling that $v_t=e^{-2f}$)
\begin{align*}
\int\Big[\Gamma_t(f)
\Delta_t(e^{-2f}g)
-((\partial_t\Gamma_t)(f)+2\Gamma_t(\Delta_t f,f))e^{-2f}g\Big]\, dm_t\geq 0.
\end{align*}
Choose $g=(\tilde g+\eps)e^{2f}$, where $\tilde g\in \Lip_b(X)\cap \D(\Delta_t)$ with $\Delta_t \tilde g\in \F$. Then $g\in\D(\Delta_t)\cap\Lip(X)$ such that $\Delta_t g\in\F$ by Lemma \ref{regofexp} and \cite[Theorem 3.4]{savare}, and there exists constants $c,C$ such that $0<c\leq g\leq C $. With this choice we obtain
\begin{align*}
\int\Big[\Gamma_t(f)
\Delta_t\tilde g
-((\partial_t\Gamma_t)(f)+2(\Delta_t f)^2\tilde g+2\Gamma_t(f,\tilde g)\Delta_t f\Big]\, dm_t\geq 0
\end{align*}
for all $f,\Delta_t f\in D(\Delta_t)\cap \Lip_b(X)$ and nonnegative $\tilde g\in \Lip_b(X)\cap \D(\Delta_t)$ with $\Delta_t \tilde g\in \F$. The result for general
$f\in D(\Delta_t)\cap \Lip_b(X)$ and nonnegative $\tilde g\in L^\infty(X)\cap \D(\Delta_t)$ follows by approximation with the standard $t$-semigroup mollifier from Definition \ref{molli}.
\end{proof}

\begin{lma}\label{lma: differentiation}
Let $(X,d,m)$ be an {\rm RCD}$(K,\infty)$-space.
Let $g\in \Lip_b(X)$ satisfying $C\geq g\geq c>0$.
Let $v\in\Lip_b(X)\cap\D(\Delta)$ such that $ \Delta^g v\in L^\infty(X)\cap\F$. Moreover let $\psi\in \cC^2(\Im(v))$. Then for $ v_r=P_r^g v$ the map
$r\mapsto \int \Gamma(u_r)\psi(u_r)g\, dm$
is absolutely continuous and
\begin{align*}
\frac{d}{dr}\int \Gamma(u_r)\psi(u_r)g\, dm=\int(2\Gamma(u_r,\Delta^g u_r)\psi(u_r)+\Gamma(u_r)\psi'(u_r)\Delta^g u_r )g\, dm 
\end{align*}
for a.e. $r\geq0$.

\end{lma}
\begin{proof}
Let $0<s<t$. Then it is well-known that, see e.g. \cite[Theorem 4.8]{eks2014} or \cite[Theorem 4.6]{gko},
\begin{align*}
&\left|\int \Gamma(v_t)\psi(v_t)g\, dm-\int \Gamma(v_s)\psi(v_s)g\, dm\right|\\
\leq& \left|\int( \Gamma(v_t)-\Gamma(v_s))\psi(v_t)g\, dm\right|+\left|\int \Gamma(v_s)(\psi(v_t)-\psi(v_s))g\, dm\right|\\
=& \left|\int_s^t\int 2\Gamma(v_r,\Delta^g v_r)\psi(v_t)g\, dm\, dr\right|+\left|\int_s^t\int \Gamma(v_s)\psi'(v_r)\Delta^g v_rg\, dm\, dr\right|\\
\leq &\norm{\psi(v_t)g}_\infty\left(\int_s^t\E^g(v_r)+\E^g(P^g_r\Delta v)\, dr\right)+(t-s)\sup_{r}\norm{\psi'(v_r)g}_\infty\, \E^g(v_s)\, \sup_r\norm{P^g_r\Delta v}_\infty\\
< &\infty,
\end{align*}
which shows $r\mapsto \int \Gamma(v_r)\psi(v_r)g\, dm$ is absolutely continuous.
We compute the a.e. derivative as follows
\begin{align*}
&\frac1h\int\left(\Gamma(v_{r+h})\psi(v_{r+h})-\Gamma(v_r)\psi(v_r) \right)g\, dm\\
=&\int\frac{\Gamma(v_{r+h})-\Gamma(v_r)}{h}\psi(v_{r+h})g\, dm
+\int \Gamma(v_r)\frac{\psi(v_{r+h})-\psi(v_r)}{h}g\, dm\\
=&\int\Gamma\left(\frac{v_{r+h}-v_r}{h},v_{r+h}+v_r\right)\psi(v_{r+h})g\, dm
+\int \Gamma(v_r)\frac{(\psi(v_{r+h})-\psi(v_r))}{h}g\, dm\\
=&-\int\frac{(v_{r+h}-v_r)}{h}\Delta^g(v_{r+h}+v_r)\psi(v_{r+h})g\, dm
-\int\psi'(v_{r+h})\Gamma(v_{r+h},v_{r+h}+v_r)\frac{(v_{r+h}-v_r)}{h}g\, dm\\
&+\int \Gamma(v_r)\frac{(\psi(v_{r+h})-\psi(v_r))}{h}g\, dm.
\end{align*}
Taking the limit $h\to0$ we verify that it holds a.e.
\begin{align*}
\lim_{h\to0}&\frac1h\int\left(\Gamma(v_{r+h})\psi(v_{r+h})-\Gamma(v_r)\psi(v_r) \right)g\, dm\\
=&-2\int \psi(v_r)(\Delta^gv_r)^2 g\, dm-2\int \psi'(v_r)\Gamma(v_r)\Delta^gv_r g\, dm 
+\int\Gamma(v_r)\psi'(v_r)\Delta^gv_r g\, dm.
\end{align*}
Applying the Leibniz and the chain rule we find that for a.e. $r\geq 0$
\begin{align*}
\frac{d}{dr}\int \Gamma(v_r)\psi(v_r)g\, dm=\int(2\Gamma(v_r,\Delta^g v_r)\psi(v_r)+\Gamma(v_r)\psi'(v_r)\Delta^g v_r )g\, dm.
\end{align*}


\end{proof}

\begin{lma}\label{regofexp}
Let $(X,d,m)$ be an {\rm RCD}$(K,\infty)$-space.
Let 
$g\in\D(\Delta)\cap \Lip(X)$ such that $\Delta g\in \F$ and $C\geq g\geq c>0$.
Let $f,\Delta f\in \D(\Delta)\cap \Lip_b(X)$. Then $e^f\in \D(\Delta)\cap \Lip_b(X)$ with $\Delta e^f,\Delta^g e^f\in L^\infty\cap \F$ and $e^f\geq c$ for some $c>0$.

Moreover the functions $t\mapsto P_te^f$ and $t\mapsto P_t^ge^f$ are continuous in $L^\infty(X)$.
\end{lma}

\begin{proof}
Since $f$ is bounded, $e^f$ is bounded as well and $e^f\geq e^{-\norm{f}_\infty}>0$. By the chain rule we have $\Gamma(e^f)=e^{2f}\Gamma(f)\in L^\infty(X)$ and
\begin{align*}
\Delta(e^f)=e^f(\Gamma(f)+\Delta f)
\end{align*}
which belongs to $L^2(X)\cap L^\infty(X)$. Next we show that $\Delta e^f\in\F$. For this note that
\begin{align*}
\E(e^f\Delta f)\leq2\int e^{2f}\Gamma(\Delta f)+(\Delta f)^2e^{2f}\Gamma(f)\, dm
\end{align*}
is bounded and
\begin{align*}
\E(e^f\Gamma(f))\leq& 2\int e^{2f} \Gamma(\Gamma(f))+\Gamma(f)^3e^{2f}\, dm\\
\leq &2\norm{e^{2f}}_\infty\int (-2K\Gamma(f)^2-\Gamma(f)\Gamma(f,\Delta f))\, dm+2\int \Gamma(f)^3e^f\, dm
\end{align*}
is bounded as well. In the last step we used \cite[Lemma 3.2]{savare} to bound $\E(\Gamma(f))$. Summing $\E(e^f\Delta f)$ and $\E(e^f\Gamma(f))$ yields that $\Delta e^f\in\F$.

Similarly we show that $\Delta^g e^f\in \D(\E^g)$. Recall first that $\D(\Delta)\subset\D(\Delta^g)$ and
\begin{align*}
\Delta^g e^f=\Delta e^f+\Gamma(\log g,e^f)
\end{align*}
which is an $L^\infty(X)$-function. Moreover note that
\begin{align*}
\E(\Delta^g e^f)=\int\Gamma(\Delta e^f)+\Gamma(\Gamma(\log g,e^f))\, dm.
\end{align*}
For the first summand we know already that it is bounded. For the second summand we use \cite[Theorem 3.4]{savare} and obtain
\begin{align*}
\int \Gamma(\Gamma(\log g,e^f))\, dm\leq 2\int (\gamma_2(\log g)-K\Gamma(\log g))\Gamma(e^f)+(\gamma_2(e^f)-K\Gamma(e^f))\Gamma(\log g)\, dm,
\end{align*}
where $\gamma_2(\log g),\gamma_2(e^f)$ are $L^1$-functions, since $\log g$ and $e^f$ belong to $\Lip_b(X)\cap \D(\Delta)$ with $\Delta \log g, \Delta e^f\in \F$.

For the last claim, note that 
\begin{align*}
P_te^f-P_se^f=\int_s^t\Delta P_r e^f\, dr,
\end{align*}
where the last integral has to be understood as a Bochner integral. Hence
\begin{align*}
\norm{P_te^f-P_se^f}_\infty=\norm{\int_s^t\Delta P_r e^f\, dr}_\infty\leq \int_s^t\norm{\Delta e^f}_\infty\, dr\leq (t-s)\norm{\Delta e^f}_\infty.
\end{align*}
The other statement follows analogously.
\end{proof}

\bibliography{srf-KoSt-6}
\end{document}